\documentclass[11pt]{article}
\usepackage{graphicx}
\graphicspath{ {figures/} }

\usepackage{color}
\usepackage{tikz}
\usepackage{amssymb,amsmath,amsthm,enumitem}
\usepackage{algorithm}
\usepackage{hyperref}
\usepackage{bm}
\usepackage[margin=1in]{geometry}
\usepackage{fancyhdr}

\theoremstyle{definition}
\newtheorem{theorem}{Theorem}
\newtheorem{lemma}[theorem]{Lemma}
\newtheorem{definition}[theorem]{Definition}
\newtheorem{proposition}[theorem]{Proposition}

\newtheorem{alg}[theorem]{Algorithm}

\usepackage{setspace}

\newtheorem{eg}[theorem]{Example} 
\newtheorem{remark}[theorem]{Remark} 

\newtheorem*{theorem*}{Theorem}

\newcommand{\R}{\mathbb R}

\begin{document}

\thispagestyle{fancy}
\fancyhead{}
\fancyfoot{}
\renewcommand{\headrulewidth}{0pt}
\cfoot{\thepage}
\rfoot{\today}

\vskip 1cm
\begin{center}
{\Large Iterative Respacing of Polygonal Curves}
\vskip 1cm

\begin{tabular*}{1.0\textwidth}{@{\extracolsep{\fill}} ll}
Marcella Manivel & Milena Silva\\
School of Mathematics& Division of Biostatistics \\
University of Minnesota   & University of Minnesota\\
{\tt maniv013@umn.edu} & {\tt silva343@umn.edu} \\
\end{tabular*}
\end{center}

\begin{center}
{
\centering
\begin{tabular*}{1.0\textwidth}{ @{\extracolsep{\fill}} l}
Robert Thompson\\
Department of Mathematics and Statistics  \\
Carleton College  \\
{\tt rthompson@carleton.edu}  \\
\end{tabular*}
}
\end{center}



\section*{Abstract}
$\indent$ A \emph{polygonal curve} is a collection of  $m$ connected line segments specified as the linear interpolation of a list of points $\{p_0, p_1, \ldots, p_m\}$. These curves may be obtained by sampling points from an oriented curve in $\R^n$.  In applications it can be useful for this sample of points to be close to \textit{equilateral}, with equal distance between consecutive points.  We present a computationally efficient method for respacing the points of a polygonal curve and show that iteration of this method converges to an equilateral polygonal curve.

\section{Introduction}  \label{sec:intro}

We introduce a method for respacing points on a polygonal curve and investigate the behavior of polygonal curves under iteration of this respacing.  A \emph{polygonal curve} (also \emph{polygonal chain} or \emph{polygonal path}) is a collection of  $m$ connected line segments specified by a sequence of points $\{p_0, p_1, \ldots, p_m\}$ in $\mathbb{R}^n$ called \emph{vertices}. Polygonal curves arise as discrete representations of continuous curves, obtained by sampling points from a continuous curve and connecting these points via linear interpolation.  A polygonal curve for which all Euclidean distances between consecutive points, $||p_{k} - p_{k-1}||$, $k = 1, \ldots, m$,  are equal will be called an \emph{equilateral polygonal curve} (or simply an \emph{equilateral curve}).  Just as the arclength parameterization of a continuous curve has theoretical and practical usefulness, working with equilateral curves instead of arbitrarily spaced polygonal curves is advantageous for certain applications.

The considerations of this paper arose from shape comparison in the context of object reassembly and automated jigsaw puzzle solving, \cite{goosst,ho,ity}.  In this context, an arbitrarily spaced polygonal curve is obtained via scanning and image processing.  This polygonal curve approximates the shape of some smoother underlying object.  It is then desirable to find an equilateral curve which approximates the original polygonal curve (and thus the underlying shape).  In \cite{ho,ity} this was accomplished by \textit{respacing} the original polygonal curve: a new set of vertices is chosen by sampling along the polygonal curve according to some fixed arclength,  and these new vertices define a polygonal curve which is empirically closer to being equilateral (see Figure \ref{fig:intro}).  This respacing process is then repeated until the resulting curve is close enough to equilateral to be used for shape comparison.  This paper makes a careful study of this repeated respacing process.

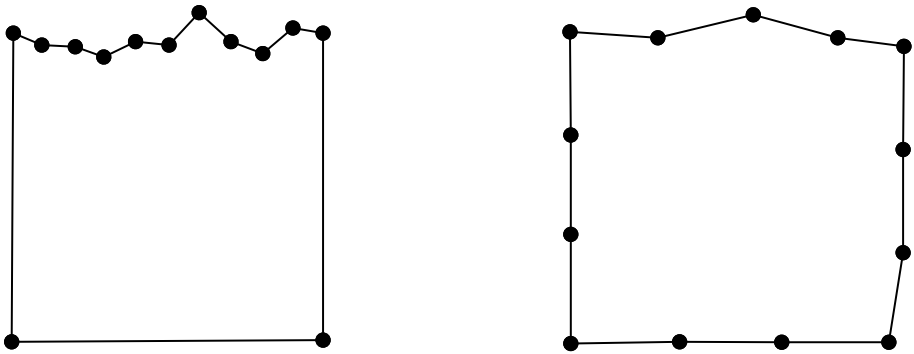
\begin{figure}[h]
    \centering

\tikzset{every picture/.style={line width=0.75pt}} 

\begin{tikzpicture}[x=0.75pt,y=0.75pt,yscale=-1,xscale=1]

\draw    (125,231.64) -- (125.84,75.83) ;
\draw [shift={(125.84,75.83)}, rotate = 270.31] [color={rgb, 255:red, 0; green, 0; blue, 0 }  ][fill={rgb, 255:red, 0; green, 0; blue, 0 }  ][line width=0.75]      (0, 0) circle [x radius= 3.35, y radius= 3.35]   ;
\draw [shift={(125,231.64)}, rotate = 270.31] [color={rgb, 255:red, 0; green, 0; blue, 0 }  ][fill={rgb, 255:red, 0; green, 0; blue, 0 }  ][line width=0.75]      (0, 0) circle [x radius= 3.35, y radius= 3.35]   ;
\draw    (282.04,230.78) -- (125,231.64) ;
\draw [shift={(125,231.64)}, rotate = 179.69] [color={rgb, 255:red, 0; green, 0; blue, 0 }  ][fill={rgb, 255:red, 0; green, 0; blue, 0 }  ][line width=0.75]      (0, 0) circle [x radius= 3.35, y radius= 3.35]   ;
\draw [shift={(282.04,230.78)}, rotate = 179.69] [color={rgb, 255:red, 0; green, 0; blue, 0 }  ][fill={rgb, 255:red, 0; green, 0; blue, 0 }  ][line width=0.75]      (0, 0) circle [x radius= 3.35, y radius= 3.35]   ;
\draw    (282.04,230.78) -- (282.04,75.83) ;
\draw [shift={(282.04,75.83)}, rotate = 270] [color={rgb, 255:red, 0; green, 0; blue, 0 }  ][fill={rgb, 255:red, 0; green, 0; blue, 0 }  ][line width=0.75]      (0, 0) circle [x radius= 3.35, y radius= 3.35]   ;
\draw [shift={(282.04,230.78)}, rotate = 270] [color={rgb, 255:red, 0; green, 0; blue, 0 }  ][fill={rgb, 255:red, 0; green, 0; blue, 0 }  ][line width=0.75]      (0, 0) circle [x radius= 3.35, y radius= 3.35]   ;
\draw    (140.2,81.86) -- (125.84,75.83) ;
\draw [shift={(125.84,75.83)}, rotate = 202.77] [color={rgb, 255:red, 0; green, 0; blue, 0 }  ][fill={rgb, 255:red, 0; green, 0; blue, 0 }  ][line width=0.75]      (0, 0) circle [x radius= 3.35, y radius= 3.35]   ;
\draw [shift={(140.2,81.86)}, rotate = 202.77] [color={rgb, 255:red, 0; green, 0; blue, 0 }  ][fill={rgb, 255:red, 0; green, 0; blue, 0 }  ][line width=0.75]      (0, 0) circle [x radius= 3.35, y radius= 3.35]   ;
\draw    (157.08,82.72) -- (140.2,81.86) ;
\draw [shift={(140.2,81.86)}, rotate = 182.92] [color={rgb, 255:red, 0; green, 0; blue, 0 }  ][fill={rgb, 255:red, 0; green, 0; blue, 0 }  ][line width=0.75]      (0, 0) circle [x radius= 3.35, y radius= 3.35]   ;
\draw [shift={(157.08,82.72)}, rotate = 182.92] [color={rgb, 255:red, 0; green, 0; blue, 0 }  ][fill={rgb, 255:red, 0; green, 0; blue, 0 }  ][line width=0.75]      (0, 0) circle [x radius= 3.35, y radius= 3.35]   ;
\draw    (171.44,87.88) -- (157.08,82.72) ;
\draw [shift={(157.08,82.72)}, rotate = 199.79] [color={rgb, 255:red, 0; green, 0; blue, 0 }  ][fill={rgb, 255:red, 0; green, 0; blue, 0 }  ][line width=0.75]      (0, 0) circle [x radius= 3.35, y radius= 3.35]   ;
\draw [shift={(171.44,87.88)}, rotate = 199.79] [color={rgb, 255:red, 0; green, 0; blue, 0 }  ][fill={rgb, 255:red, 0; green, 0; blue, 0 }  ][line width=0.75]      (0, 0) circle [x radius= 3.35, y radius= 3.35]   ;
\draw    (187.48,80.13) -- (171.44,87.88) ;
\draw [shift={(171.44,87.88)}, rotate = 154.22] [color={rgb, 255:red, 0; green, 0; blue, 0 }  ][fill={rgb, 255:red, 0; green, 0; blue, 0 }  ][line width=0.75]      (0, 0) circle [x radius= 3.35, y radius= 3.35]   ;
\draw [shift={(187.48,80.13)}, rotate = 154.22] [color={rgb, 255:red, 0; green, 0; blue, 0 }  ][fill={rgb, 255:red, 0; green, 0; blue, 0 }  ][line width=0.75]      (0, 0) circle [x radius= 3.35, y radius= 3.35]   ;
\draw    (204.36,81.86) -- (187.48,80.13) ;
\draw [shift={(187.48,80.13)}, rotate = 185.82] [color={rgb, 255:red, 0; green, 0; blue, 0 }  ][fill={rgb, 255:red, 0; green, 0; blue, 0 }  ][line width=0.75]      (0, 0) circle [x radius= 3.35, y radius= 3.35]   ;
\draw [shift={(204.36,81.86)}, rotate = 185.82] [color={rgb, 255:red, 0; green, 0; blue, 0 }  ][fill={rgb, 255:red, 0; green, 0; blue, 0 }  ][line width=0.75]      (0, 0) circle [x radius= 3.35, y radius= 3.35]   ;
\draw    (219.56,65.5) -- (204.36,81.86) ;
\draw [shift={(204.36,81.86)}, rotate = 132.9] [color={rgb, 255:red, 0; green, 0; blue, 0 }  ][fill={rgb, 255:red, 0; green, 0; blue, 0 }  ][line width=0.75]      (0, 0) circle [x radius= 3.35, y radius= 3.35]   ;
\draw [shift={(219.56,65.5)}, rotate = 132.9] [color={rgb, 255:red, 0; green, 0; blue, 0 }  ][fill={rgb, 255:red, 0; green, 0; blue, 0 }  ][line width=0.75]      (0, 0) circle [x radius= 3.35, y radius= 3.35]   ;
\draw    (235.6,80.13) -- (219.56,65.5) ;
\draw [shift={(219.56,65.5)}, rotate = 222.37] [color={rgb, 255:red, 0; green, 0; blue, 0 }  ][fill={rgb, 255:red, 0; green, 0; blue, 0 }  ][line width=0.75]      (0, 0) circle [x radius= 3.35, y radius= 3.35]   ;
\draw [shift={(235.6,80.13)}, rotate = 222.37] [color={rgb, 255:red, 0; green, 0; blue, 0 }  ][fill={rgb, 255:red, 0; green, 0; blue, 0 }  ][line width=0.75]      (0, 0) circle [x radius= 3.35, y radius= 3.35]   ;
\draw    (251.64,86.16) -- (235.6,80.13) ;
\draw [shift={(235.6,80.13)}, rotate = 200.59] [color={rgb, 255:red, 0; green, 0; blue, 0 }  ][fill={rgb, 255:red, 0; green, 0; blue, 0 }  ][line width=0.75]      (0, 0) circle [x radius= 3.35, y radius= 3.35]   ;
\draw [shift={(251.64,86.16)}, rotate = 200.59] [color={rgb, 255:red, 0; green, 0; blue, 0 }  ][fill={rgb, 255:red, 0; green, 0; blue, 0 }  ][line width=0.75]      (0, 0) circle [x radius= 3.35, y radius= 3.35]   ;
\draw    (266.84,73.25) -- (251.64,86.16) ;
\draw [shift={(251.64,86.16)}, rotate = 139.65] [color={rgb, 255:red, 0; green, 0; blue, 0 }  ][fill={rgb, 255:red, 0; green, 0; blue, 0 }  ][line width=0.75]      (0, 0) circle [x radius= 3.35, y radius= 3.35]   ;
\draw [shift={(266.84,73.25)}, rotate = 139.65] [color={rgb, 255:red, 0; green, 0; blue, 0 }  ][fill={rgb, 255:red, 0; green, 0; blue, 0 }  ][line width=0.75]      (0, 0) circle [x radius= 3.35, y radius= 3.35]   ;
\draw    (282.04,75.83) -- (266.84,73.25) ;
\draw [shift={(266.84,73.25)}, rotate = 189.64] [color={rgb, 255:red, 0; green, 0; blue, 0 }  ][fill={rgb, 255:red, 0; green, 0; blue, 0 }  ][line width=0.75]      (0, 0) circle [x radius= 3.35, y radius= 3.35]   ;
\draw [shift={(282.04,75.83)}, rotate = 189.64] [color={rgb, 255:red, 0; green, 0; blue, 0 }  ][fill={rgb, 255:red, 0; green, 0; blue, 0 }  ][line width=0.75]      (0, 0) circle [x radius= 3.35, y radius= 3.35]   ;
\draw    (406.99,232.5) -- (406.99,177.41) ;
\draw [shift={(406.99,177.41)}, rotate = 270] [color={rgb, 255:red, 0; green, 0; blue, 0 }  ][fill={rgb, 255:red, 0; green, 0; blue, 0 }  ][line width=0.75]      (0, 0) circle [x radius= 3.35, y radius= 3.35]   ;
\draw [shift={(406.99,232.5)}, rotate = 270] [color={rgb, 255:red, 0; green, 0; blue, 0 }  ][fill={rgb, 255:red, 0; green, 0; blue, 0 }  ][line width=0.75]      (0, 0) circle [x radius= 3.35, y radius= 3.35]   ;
\draw    (461.87,231.64) -- (406.99,232.5) ;
\draw [shift={(406.99,232.5)}, rotate = 179.1] [color={rgb, 255:red, 0; green, 0; blue, 0 }  ][fill={rgb, 255:red, 0; green, 0; blue, 0 }  ][line width=0.75]      (0, 0) circle [x radius= 3.35, y radius= 3.35]   ;
\draw [shift={(461.87,231.64)}, rotate = 179.1] [color={rgb, 255:red, 0; green, 0; blue, 0 }  ][fill={rgb, 255:red, 0; green, 0; blue, 0 }  ][line width=0.75]      (0, 0) circle [x radius= 3.35, y radius= 3.35]   ;
\draw    (567.4,231.85) -- (574.58,186.66) ;
\draw [shift={(574.58,186.66)}, rotate = 279.02] [color={rgb, 255:red, 0; green, 0; blue, 0 }  ][fill={rgb, 255:red, 0; green, 0; blue, 0 }  ][line width=0.75]      (0, 0) circle [x radius= 3.35, y radius= 3.35]   ;
\draw [shift={(567.4,231.85)}, rotate = 279.02] [color={rgb, 255:red, 0; green, 0; blue, 0 }  ][fill={rgb, 255:red, 0; green, 0; blue, 0 }  ][line width=0.75]      (0, 0) circle [x radius= 3.35, y radius= 3.35]   ;
\draw    (513.37,231.85) -- (461.87,231.64) ;
\draw [shift={(461.87,231.64)}, rotate = 180.24] [color={rgb, 255:red, 0; green, 0; blue, 0 }  ][fill={rgb, 255:red, 0; green, 0; blue, 0 }  ][line width=0.75]      (0, 0) circle [x radius= 3.35, y radius= 3.35]   ;
\draw [shift={(513.37,231.85)}, rotate = 180.24] [color={rgb, 255:red, 0; green, 0; blue, 0 }  ][fill={rgb, 255:red, 0; green, 0; blue, 0 }  ][line width=0.75]      (0, 0) circle [x radius= 3.35, y radius= 3.35]   ;
\draw    (567.4,231.85) -- (513.37,231.85) ;
\draw [shift={(513.37,231.85)}, rotate = 180] [color={rgb, 255:red, 0; green, 0; blue, 0 }  ][fill={rgb, 255:red, 0; green, 0; blue, 0 }  ][line width=0.75]      (0, 0) circle [x radius= 3.35, y radius= 3.35]   ;
\draw [shift={(567.4,231.85)}, rotate = 180] [color={rgb, 255:red, 0; green, 0; blue, 0 }  ][fill={rgb, 255:red, 0; green, 0; blue, 0 }  ][line width=0.75]      (0, 0) circle [x radius= 3.35, y radius= 3.35]   ;
\draw    (406.99,127.26) -- (406.99,177.41) ;
\draw [shift={(406.99,177.41)}, rotate = 90] [color={rgb, 255:red, 0; green, 0; blue, 0 }  ][fill={rgb, 255:red, 0; green, 0; blue, 0 }  ][line width=0.75]      (0, 0) circle [x radius= 3.35, y radius= 3.35]   ;
\draw [shift={(406.99,127.26)}, rotate = 90] [color={rgb, 255:red, 0; green, 0; blue, 0 }  ][fill={rgb, 255:red, 0; green, 0; blue, 0 }  ][line width=0.75]      (0, 0) circle [x radius= 3.35, y radius= 3.35]   ;
\draw    (406.57,75.18) -- (406.99,127.26) ;
\draw [shift={(406.99,127.26)}, rotate = 89.54] [color={rgb, 255:red, 0; green, 0; blue, 0 }  ][fill={rgb, 255:red, 0; green, 0; blue, 0 }  ][line width=0.75]      (0, 0) circle [x radius= 3.35, y radius= 3.35]   ;
\draw [shift={(406.57,75.18)}, rotate = 89.54] [color={rgb, 255:red, 0; green, 0; blue, 0 }  ][fill={rgb, 255:red, 0; green, 0; blue, 0 }  ][line width=0.75]      (0, 0) circle [x radius= 3.35, y radius= 3.35]   ;
\draw    (450.89,78.2) -- (406.57,75.18) ;
\draw [shift={(406.57,75.18)}, rotate = 183.89] [color={rgb, 255:red, 0; green, 0; blue, 0 }  ][fill={rgb, 255:red, 0; green, 0; blue, 0 }  ][line width=0.75]      (0, 0) circle [x radius= 3.35, y radius= 3.35]   ;
\draw [shift={(450.89,78.2)}, rotate = 183.89] [color={rgb, 255:red, 0; green, 0; blue, 0 }  ][fill={rgb, 255:red, 0; green, 0; blue, 0 }  ][line width=0.75]      (0, 0) circle [x radius= 3.35, y radius= 3.35]   ;
\draw    (499.02,66.58) -- (450.89,78.2) ;
\draw [shift={(450.89,78.2)}, rotate = 166.42] [color={rgb, 255:red, 0; green, 0; blue, 0 }  ][fill={rgb, 255:red, 0; green, 0; blue, 0 }  ][line width=0.75]      (0, 0) circle [x radius= 3.35, y radius= 3.35]   ;
\draw [shift={(499.02,66.58)}, rotate = 166.42] [color={rgb, 255:red, 0; green, 0; blue, 0 }  ][fill={rgb, 255:red, 0; green, 0; blue, 0 }  ][line width=0.75]      (0, 0) circle [x radius= 3.35, y radius= 3.35]   ;
\draw    (541.65,78.2) -- (499.02,66.58) ;
\draw [shift={(499.02,66.58)}, rotate = 195.25] [color={rgb, 255:red, 0; green, 0; blue, 0 }  ][fill={rgb, 255:red, 0; green, 0; blue, 0 }  ][line width=0.75]      (0, 0) circle [x radius= 3.35, y radius= 3.35]   ;
\draw [shift={(541.65,78.2)}, rotate = 195.25] [color={rgb, 255:red, 0; green, 0; blue, 0 }  ][fill={rgb, 255:red, 0; green, 0; blue, 0 }  ][line width=0.75]      (0, 0) circle [x radius= 3.35, y radius= 3.35]   ;
\draw    (575,82.5) -- (541.65,78.2) ;
\draw [shift={(541.65,78.2)}, rotate = 187.35] [color={rgb, 255:red, 0; green, 0; blue, 0 }  ][fill={rgb, 255:red, 0; green, 0; blue, 0 }  ][line width=0.75]      (0, 0) circle [x radius= 3.35, y radius= 3.35]   ;
\draw [shift={(575,82.5)}, rotate = 187.35] [color={rgb, 255:red, 0; green, 0; blue, 0 }  ][fill={rgb, 255:red, 0; green, 0; blue, 0 }  ][line width=0.75]      (0, 0) circle [x radius= 3.35, y radius= 3.35]   ;
\draw    (575,82.5) -- (574.58,134.58) ;
\draw [shift={(574.58,134.58)}, rotate = 90.46] [color={rgb, 255:red, 0; green, 0; blue, 0 }  ][fill={rgb, 255:red, 0; green, 0; blue, 0 }  ][line width=0.75]      (0, 0) circle [x radius= 3.35, y radius= 3.35]   ;
\draw [shift={(575,82.5)}, rotate = 90.46] [color={rgb, 255:red, 0; green, 0; blue, 0 }  ][fill={rgb, 255:red, 0; green, 0; blue, 0 }  ][line width=0.75]      (0, 0) circle [x radius= 3.35, y radius= 3.35]   ;
\draw    (574.58,186.66) -- (574.58,134.58) ;
\draw [shift={(574.58,134.58)}, rotate = 270] [color={rgb, 255:red, 0; green, 0; blue, 0 }  ][fill={rgb, 255:red, 0; green, 0; blue, 0 }  ][line width=0.75]      (0, 0) circle [x radius= 3.35, y radius= 3.35]   ;
\draw [shift={(574.58,186.66)}, rotate = 270] [color={rgb, 255:red, 0; green, 0; blue, 0 }  ][fill={rgb, 255:red, 0; green, 0; blue, 0 }  ][line width=0.75]      (0, 0) circle [x radius= 3.35, y radius= 3.35]   ;

\end{tikzpicture}
    \caption{An unevenly spaced polygonal curve (left) and the polygonal curve that results after one application of arclength respacing (right).}
    \label{fig:intro}
\end{figure}

In Section \ref{sec:method} we introduce the arclength respacing process for a polygonal curve and establish some of its key properties.  In Section \ref{sec:iteration} we consider iteration of arclength respacing and show in Theorem \ref{thm:main} that the iteration tends toward a limiting polygonal curve and that the limit is equilateral.  In Section \ref{sec:examples} we examine various examples. In Section $\ref{sec:application}$ we outline the implementation of arclength respacing and apply it to polygonal curves approximating a shape.

\section{The arclength respacing method} \label{sec:method}

Let $C$ be a polygonal curve with vertices $\{p_0, p_1, p_2, \ldots, p_m\}$.  Curves will be visualized in $\mathbb{R}^2$, but all discussion that follows applies equally to curves in $\mathbb{R}^n$.  We make no restrictions on the points that define $C$; for example, polygonal curves may be closed (have an overlapping startpoint and endpoint) or otherwise have self-intersections or overlapping vertices. Denote by $\displaystyle L(C) = \sum_{k=1}^m ||p_k - p_{k-1} ||$ the length of $C$. Let $P(s)$ be the linear interpolation of the vertices of $C$, parameterized by arclength $s$, such that $0 \leq s \leq L(C)$.  We now define the process that we will study in the following sections: the arclength respacing of a polygonal curve.

\begin{definition}
The \emph{arclength respacing} of the polygonal curve $C$ is the polygonal curve $f(C)$ with vertices $\{f(p_0), \ldots, f(p_m) \}$ evenly spaced by arclength along $C$,
\[
f(p_k) = P\left ( k \frac{L(C)}{m} \right), \quad k = 0, \ldots, m.
\]
\end{definition}

\noindent
An example result of arclength respacing is shown in Figure \ref{fig:respacing}.

\begin{figure}[ht]
    \centering
\tikzset{every picture/.style={line width=0.75pt}} 

\begin{tikzpicture}[x=0.75pt,y=0.75pt,yscale=-1,xscale=1]

\draw    (111.98,274.77) -- (114.54,164.89) ;
\draw [shift={(114.54,164.89)}, rotate = 271.34] [color={rgb, 255:red, 0; green, 0; blue, 0 }  ][fill={rgb, 255:red, 0; green, 0; blue, 0 }  ][line width=0.75]      (0, 0) circle [x radius= 3.35, y radius= 3.35]   ;
\draw [shift={(111.98,274.77)}, rotate = 271.34] [color={rgb, 255:red, 0; green, 0; blue, 0 }  ][fill={rgb, 255:red, 0; green, 0; blue, 0 }  ][line width=0.75]      (0, 0) circle [x radius= 3.35, y radius= 3.35]   ;
\draw    (114.54,164.89) -- (389,119) ;
\draw [shift={(389,119)}, rotate = 350.51] [color={rgb, 255:red, 0; green, 0; blue, 0 }  ][fill={rgb, 255:red, 0; green, 0; blue, 0 }  ][line width=0.75]      (0, 0) circle [x radius= 3.35, y radius= 3.35]   ;
\draw [shift={(114.54,164.89)}, rotate = 350.51] [color={rgb, 255:red, 0; green, 0; blue, 0 }  ][fill={rgb, 255:red, 0; green, 0; blue, 0 }  ][line width=0.75]      (0, 0) circle [x radius= 3.35, y radius= 3.35]   ;
\draw [color={rgb, 255:red, 0; green, 0; blue, 0 }  ,draw opacity=0.4 ][line width=1.5]  [dash pattern={on 5.63pt off 4.5pt}]  (111.98,274.77) -- (143.33,160.17) ;
\draw [shift={(111.98,274.77)}, rotate = 285.3] [color={rgb, 255:red, 0; green, 0; blue, 0 }  ,draw opacity=0.4 ][fill={rgb, 255:red, 0; green, 0; blue, 0 }  ,fill opacity=0.4 ][line width=1.5]      (0, 0) circle [x radius= 4.36, y radius= 4.36]   ;
\draw    (389,119) -- (441,168) ;
\draw [shift={(441,168)}, rotate = 43.3] [color={rgb, 255:red, 0; green, 0; blue, 0 }  ][fill={rgb, 255:red, 0; green, 0; blue, 0 }  ][line width=0.75]      (0, 0) circle [x radius= 3.35, y radius= 3.35]   ;
\draw [shift={(389,119)}, rotate = 43.3] [color={rgb, 255:red, 0; green, 0; blue, 0 }  ][fill={rgb, 255:red, 0; green, 0; blue, 0 }  ][line width=0.75]      (0, 0) circle [x radius= 3.35, y radius= 3.35]   ;
\draw    (441,168) -- (455,87) ;
\draw [shift={(455,87)}, rotate = 279.81] [color={rgb, 255:red, 0; green, 0; blue, 0 }  ][fill={rgb, 255:red, 0; green, 0; blue, 0 }  ][line width=0.75]      (0, 0) circle [x radius= 3.35, y radius= 3.35]   ;
\draw [shift={(441,168)}, rotate = 279.81] [color={rgb, 255:red, 0; green, 0; blue, 0 }  ][fill={rgb, 255:red, 0; green, 0; blue, 0 }  ][line width=0.75]      (0, 0) circle [x radius= 3.35, y radius= 3.35]   ;
\draw [color={rgb, 255:red, 0; green, 0; blue, 0 }  ,draw opacity=0.4 ][line width=1.5]  [dash pattern={on 5.63pt off 4.5pt}]  (143.33,160.17) -- (288.67,135.83) ;
\draw [shift={(288.67,135.83)}, rotate = 350.5] [color={rgb, 255:red, 0; green, 0; blue, 0 }  ,draw opacity=0.4 ][fill={rgb, 255:red, 0; green, 0; blue, 0 }  ,fill opacity=0.4 ][line width=1.5]      (0, 0) circle [x radius= 4.36, y radius= 4.36]   ;
\draw [shift={(143.33,160.17)}, rotate = 350.5] [color={rgb, 255:red, 0; green, 0; blue, 0 }  ,draw opacity=0.4 ][fill={rgb, 255:red, 0; green, 0; blue, 0 }  ,fill opacity=0.4 ][line width=1.5]      (0, 0) circle [x radius= 4.36, y radius= 4.36]   ;
\draw [color={rgb, 255:red, 0; green, 0; blue, 0 }  ,draw opacity=0.4 ][line width=1.5]  [dash pattern={on 5.63pt off 4.5pt}]  (288.67,135.83) -- (410,137) ;
\draw [color={rgb, 255:red, 0; green, 0; blue, 0 }  ,draw opacity=0.4 ][line width=1.5]  [dash pattern={on 5.63pt off 4.5pt}]  (410,137) -- (455,87) ;
\draw [shift={(455,87)}, rotate = 311.99] [color={rgb, 255:red, 0; green, 0; blue, 0 }  ,draw opacity=0.4 ][fill={rgb, 255:red, 0; green, 0; blue, 0 }  ,fill opacity=0.4 ][line width=1.5]      (0, 0) circle [x radius= 4.36, y radius= 4.36]   ;
\draw [shift={(410,137)}, rotate = 311.99] [color={rgb, 255:red, 0; green, 0; blue, 0 }  ,draw opacity=0.4 ][fill={rgb, 255:red, 0; green, 0; blue, 0 }  ,fill opacity=0.4 ][line width=1.5]      (0, 0) circle [x radius= 4.36, y radius= 4.36]   ;

\draw (79.52,265.2) node [anchor=north west][inner sep=0.75pt]    {$p_{0}$};
\draw (95.24,140.04) node [anchor=north west][inner sep=0.75pt]    {$p_{1}$};
\draw (277.6,105.08) node [anchor=north west][inner sep=0.75pt]    {$p_{2}$};
\draw (447.66,150.97) node [anchor=north west][inner sep=0.75pt]    {$p_{4}$};
\draw (379.73,91.67) node [anchor=north west][inner sep=0.75pt]    {$p_{3}$};
\draw (122.84,266.51) node [anchor=north west][inner sep=0.75pt]    {$f( p_{0})$};
\draw (153.06,167.74) node [anchor=north west][inner sep=0.75pt]    {$f( p_{1})$};
\draw (276.37,146.62) node [anchor=north west][inner sep=0.75pt]    {$f( p_{2})$};
\draw (467.03,82.62) node [anchor=north west][inner sep=0.75pt]    {$f( p_{4})$};
\draw (367.7,143.62) node [anchor=north west][inner sep=0.75pt]    {$f( p_{3})$};

\end{tikzpicture}

    \caption{A polygonal curve $C$ and its arclength respacing $f(C)$.}
    \label{fig:respacing}
\end{figure}
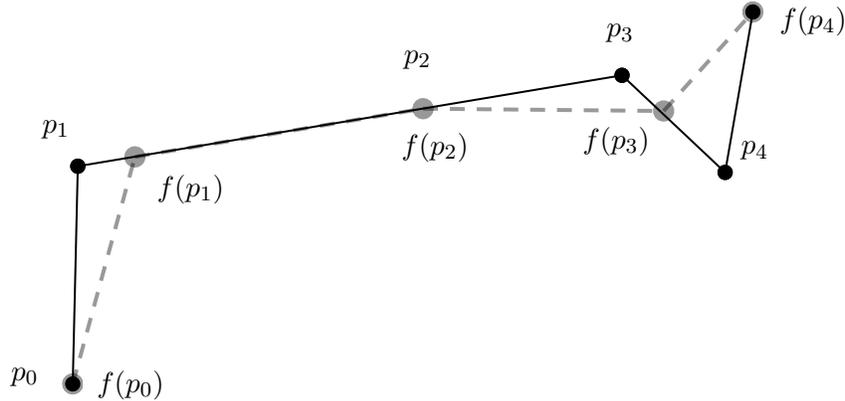

It is important to note for applications that arclength respacing can be performed computationally without costly numerical integration or function inversion typically needed to find an arclength parameterization.  The key idea is that piecewise linearity can be exploited to find $P(s)$.  Let $d_0 = 0$ and $\displaystyle d_k = \sum_{i=1}^{k} || p_{i}-p_{i-1} ||$ the arclength of $C$ up to vertex $d_k$ for $k=1, \ldots, m$.  Then, when $d_k \neq 0$, for $0 \leq t \leq 1$ and $k = 1, \ldots, m$, the parameterizations of segments
$$
P\big ( (1-t) d_{k-1} + t d_k \big) = (1-t) p_{k-1} + t p_{k}
$$
can be combined to provide an arclength parameterization $P(s)$ of $C$.  Thus it is possible to compute $f(C)$ from $C$ quickly using only linear interpolation.  Section \ref{sec:application} outlines this process in more detail.

Polygonal curves with different vertex sets can have the same arclength parameterized linear interpolation.  To account for this ambiguity, we introduce the notion of \textit{similar} polygonal curves.

\begin{definition}
A vertex $p_k$ of a polygonal curve is called a \emph{basic vertex} if it is not equal to the preceding vertex $p_{k-1}$ and does not lie on the line segment between its neighboring vertices $p_{k-1}$ and $p_{k+1}$.  Polygonal curves $P$ and $Q$ are \emph{similar}, $P \sim Q$, if they share the same sequence of basic vertices.  Note that the initial vertex $p_0$ is always basic.
\end{definition}
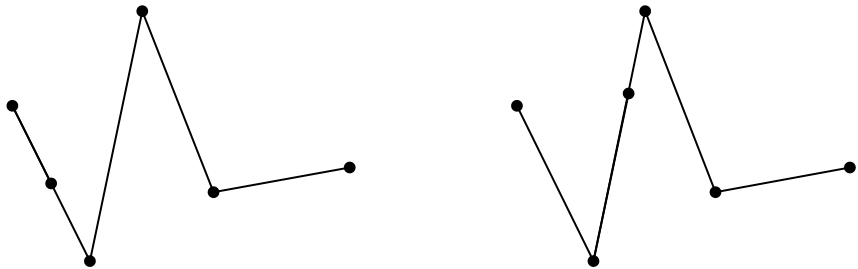
\begin{figure}[ht]
    \centering

\tikzset{every picture/.style={line width=0.75pt}} 

\begin{tikzpicture}[x=0.55pt,y=0.55pt,yscale=-1,xscale=1]

\draw    (139.34,53.5) -- (188.34,178.08) ;
\draw [shift={(139.34,53.5)}, rotate = 68.53] [color={rgb, 255:red, 0; green, 0; blue, 0 }  ][fill={rgb, 255:red, 0; green, 0; blue, 0 }  ][line width=0.75]      (0, 0) circle [x radius= 3.35, y radius= 3.35]   ;
\draw    (188.34,178.08) -- (282,161.09) ;
\draw [shift={(282,161.09)}, rotate = 349.72] [color={rgb, 255:red, 0; green, 0; blue, 0 }  ][fill={rgb, 255:red, 0; green, 0; blue, 0 }  ][line width=0.75]      (0, 0) circle [x radius= 3.35, y radius= 3.35]   ;
\draw [shift={(188.34,178.08)}, rotate = 349.72] [color={rgb, 255:red, 0; green, 0; blue, 0 }  ][fill={rgb, 255:red, 0; green, 0; blue, 0 }  ][line width=0.75]      (0, 0) circle [x radius= 3.35, y radius= 3.35]   ;
\draw    (50,118.62) -- (103.32,225.5) ;
\draw [shift={(50,118.62)}, rotate = 63.49] [color={rgb, 255:red, 0; green, 0; blue, 0 }  ][fill={rgb, 255:red, 0; green, 0; blue, 0 }  ][line width=0.75]      (0, 0) circle [x radius= 3.35, y radius= 3.35]   ;
\draw    (103.32,225.5) -- (139.34,53.5) ;
\draw [shift={(103.32,225.5)}, rotate = 281.83] [color={rgb, 255:red, 0; green, 0; blue, 0 }  ][fill={rgb, 255:red, 0; green, 0; blue, 0 }  ][line width=0.75]      (0, 0) circle [x radius= 3.35, y radius= 3.35]   ;
\draw    (485.19,53.5) -- (533.55,178.08) ;
\draw [shift={(485.19,53.5)}, rotate = 68.78] [color={rgb, 255:red, 0; green, 0; blue, 0 }  ][fill={rgb, 255:red, 0; green, 0; blue, 0 }  ][line width=0.75]      (0, 0) circle [x radius= 3.35, y radius= 3.35]   ;
\draw    (533.55,178.08) -- (626,161.09) ;
\draw [shift={(626,161.09)}, rotate = 349.59] [color={rgb, 255:red, 0; green, 0; blue, 0 }  ][fill={rgb, 255:red, 0; green, 0; blue, 0 }  ][line width=0.75]      (0, 0) circle [x radius= 3.35, y radius= 3.35]   ;
\draw [shift={(533.55,178.08)}, rotate = 349.59] [color={rgb, 255:red, 0; green, 0; blue, 0 }  ][fill={rgb, 255:red, 0; green, 0; blue, 0 }  ][line width=0.75]      (0, 0) circle [x radius= 3.35, y radius= 3.35]   ;
\draw    (397,118.62) -- (449.63,225.5) ;
\draw [shift={(397,118.62)}, rotate = 63.78] [color={rgb, 255:red, 0; green, 0; blue, 0 }  ][fill={rgb, 255:red, 0; green, 0; blue, 0 }  ][line width=0.75]      (0, 0) circle [x radius= 3.35, y radius= 3.35]   ;
\draw    (449.63,225.5) -- (485.19,53.5) ;
\draw [shift={(449.63,225.5)}, rotate = 281.68] [color={rgb, 255:red, 0; green, 0; blue, 0 }  ][fill={rgb, 255:red, 0; green, 0; blue, 0 }  ][line width=0.75]      (0, 0) circle [x radius= 3.35, y radius= 3.35]   ;
\draw    (50,118.62) -- (76.66,172.06) ;
\draw [shift={(76.66,172.06)}, rotate = 63.49] [color={rgb, 255:red, 0; green, 0; blue, 0 }  ][fill={rgb, 255:red, 0; green, 0; blue, 0 }  ][line width=0.75]      (0, 0) circle [x radius= 3.35, y radius= 3.35]   ;
\draw    (449.63,225.5) -- (473.81,110.13) ;
\draw [shift={(473.81,110.13)}, rotate = 281.84] [color={rgb, 255:red, 0; green, 0; blue, 0 }  ][fill={rgb, 255:red, 0; green, 0; blue, 0 }  ][line width=0.75]      (0, 0) circle [x radius= 3.35, y radius= 3.35]   ;

\end{tikzpicture}

    \caption{Two similar polygonal curves with different vertices.}
    \label{fig:similarpc}
\end{figure}

\begin{lemma} \label{lem:simlength}
Similar polygonal curves have the same arclength parameterized linear interpolation.
\end{lemma}
\begin{proof}
If $p_k$ is a non-basic vertex, either $p_{k-1} = p_k$, or $p_k$ lies in the image of the line segment connecting $p_{k-1}$ and $p_{k+1}$.   In either case, removing $p_k$ from the vertex set of $C$ will not affect the arclength parameterized linear interpolation (we  either remove a segment of length zero or replace two segments with one segment of combined length).  Thus the arclength parameterized linear interpolation is determined only by the basic vertices, and similar curves have the same basic vertices.
\end{proof}

We now investigate how the process of arclength respacing changes a polygonal curve.  In particular, we find in Proposition \ref{prop:equivprops} that similarity characterizes particular ways that the length and spacing of a polygonal curve change under arclength respacing.  We first prove a simple lemma about a more general respacing process we call \textit{oriented resampling}, of which arclength respacing is a special case.

\begin{definition} \label{def:orientedrespacing}
Let $C$ be a polygonal curve and $P(s)$, $0 \leq s \leq L(C)$, the associated arclength parameterized linear interpolation of the vertices of $C$.  Suppose we sample $m+1$ values of arclength $0=s_0 \leq s_1 \leq \cdots \leq s_{m-1} \leq s_m = L(C)$ and define a corresponding sample of points $q_0, \ldots, q_m$ along $C$ such that $q_k = P(s_k)$, for $k=0, \ldots, m$.  The polygonal curve $D$ defined by the vertex set ${q_0, \ldots, q_m}$ is called an \textit{oriented resampling} of $P$.
\end{definition}

\begin{lemma} \label{lem:length}
If $D$ is an oriented resampling of $C$, then $L(D) \leq L(C)$.  In particular, $L(f(C)) \leq L(C)$, so arclength respacing cannot increase the length of a polygonal curve.
\end{lemma}

\begin{proof}
We use the notation of Definition \ref{def:orientedrespacing}.  First observe that $||q_k - q_{k-1}|| \leq s_k - s_{k-1}$, since $s_k - s_{k-1}$  is the arclength distance between points $q_{k-1}$ and $q_k$ measured along $C$.
Thus
\[
\begin{split}
L(D) &= \sum_{k=1}^m ||q_k - q_{k-1} || \\
& \leq \sum_{k=1}^m s_k - s_{k-1} \\
& = s_m - s_0 = L(C).
\end{split}
\]
\end{proof}

\begin{proposition}  \label{prop:equivprops}
Let $C$ be a polygonal curve and $f(C)$ the arclength respacing of $C$.  The following properties are equivalent:
\begin{enumerate}
\item[(1)] $L(C) = L(f(C)),$
\item[(2)] $C \sim f(C),$
\item[(3)] $f(C)$ is equilateral.
\end{enumerate}
\end{proposition}

\begin{proof}

We first show $(1) \iff (2)$.  If $C \sim f(C)$ then $L(f(C)) = L(C)$ by Lemma \ref{lem:simlength}.  Suppose then that  $f(C) \not \sim C$.  There must then be a vertex $p_k$ of $C$ which is not equal to any vertex of $f(C)$, nor lies on a line segment connecting two consecutive vertices of $f(C)$.  Suppose that $q_n$ and $q_{n+1}$ are the vertices of $f(C)$ immediately preceding and following $p_k$ along $C$.  Let $p_l$ be the vertex of $C$ immediately preceding $q_n$ along $C$ and $p_r$ the vertex immediately following $q_{n+1}$ along $C$.  Let 
\[
\overline C = \{ p_0, \ldots, p_l, q_n, q_{n+1}, p_r, \ldots, p_m\}
\]
be the polygonal curve obtained from $C$ by adding the vertices $q_n, q_{n+1}$ to $C$ and deleting all vertices $p_{l+1},\ldots, p_k,\ldots, p_{r-1}$ between $q_n$ and $q_{n+1}$ along $C$.  Because these deleted vertices do not all lie along the line segment connecting $q_n$ and $q_{n+1}$ (in particular, $p_k$ does not), we have the strict inequality: $L(\overline C) < L(C)$.  Since the arclength respacing $f(C)$ is also an oriented resampling of $\overline C$, $L(f(C)) \leq L(\overline C)$ by Lemma \ref{lem:length}.  Combining these two inequalities gives $L(f(C)) < L(C)$.

We next show $(2) \iff (3)$.  Suppose first that $C\sim f(C)$. Then, by Lemma \ref{lem:simlength}, $C$ and $f(C)$ have the same arclength parameterized linear interpolation.  Thus a sampling of points evenly spaced along $C$ by arclength will also be evenly spaced by arclength along $f(C)$, so $f(C)$ will be equilateral.

Assume now that $f(C)$ is equilateral.  Because there are $m+1$ points in $f(C)$ and $m$ line segments in $C$, there must be a pair of consecutive vertices $q_k,q_{k+1}$ of $f(C)$ which lie on the same line segment of $C$.  Thus $||q_{k+1}-q_k|| = L(C)/m$.  Because $f(C)$ is equilateral, we also have $||q_{k+1}-q_k|| = L(f(C))/m$.  Thus $L(C) = L(f(C))$, and so $C \sim f(C)$ since we have already established that $(1) \implies (2)$.
\end{proof}

 We now examine what happens when the arclength respacing process is repeated.  First observe that, although the vertices of $f(C)$ are equally spaced by arclength  along $C$, they are not necessarily equally spaced by arclength along $f(C)$.  Thus $f(C)$ is not necessarily an equilateral polygonal curve.  Nevertheless, it appears that vertices of $f(C)$ are more evenly spaced than those of $C$ (as seen clearly in Figure \ref{fig:intro}). Thus, in an effort to obtain an equilateral curve, we can iterate the arclength respacing process, hoping to obtain polygonal curves that are increasingly close to being equilateral. The result of iterated arclength respacing is considered in the next section.

\section{Iteration of the respacing method}  \label{sec:iteration}

Let $C$ be a polygonal curve and denote by $C^n$ the $n^{th}$ iteration of the arclength respacing:  $C^n = f^n(C)$.  Lemma \ref{lem:length} yields an immediate observation about this iteration: the sequence of lengths of the iterated curves must converge.

\begin{lemma}
\label{lem:conv}
$L(C^n)$ converges as $n \to \infty$. 
\end{lemma}

\begin{proof} From Lemma \ref{lem:length}, the sequence $\left \{ L(C^n) \right\}$ is nonincreasing.  This sequence is also bounded below by the distance $||p_m-p_0||$ between the endpoints of $C$. (Note that this could be $0$ for a closed polygonal curve.) Thus $\left \{ L(C^n) \right\}$ is bounded and monotonic and must converge.
\end{proof}

As demonstrated in Proposition \ref{prop:equivprops}, arclength respacing yields an equilateral curve only in special cases.  However, arclength respacing generally produces a curve that appears to being equilateral (as seen in Figure \ref{fig:respacing}, for example).  More precisely, as the respacing is repeated, the iterates will converge to an equilateral curve.  We introduce an equivalent definition of equilateral that will be useful in proving this.

\begin{lemma} \label{lem:prebigtheorem}
$C  =\{p_0, \ldots, p_m \}$ is equilateral if and only if, for $k = 1, \ldots, m$,
\[
\frac{k L(C)}{m} = \sum_{i=1}^k ||p_i - p_{i-1}||.
\]
\end{lemma}
\begin{theorem} \label{thm:main}  For any polygonal curve $C$, $\displaystyle \lim_{n \to \infty}f^n(C)$ exists and is an equilateral curve.
\end{theorem}

\begin{proof}
We will argue that, for each $p_k$, the sequence $\{f^n(p_k)\}_{n=0}^\infty$ converges, and thus that the limiting polygonal curve is given by $C^* = \{p^*_0, \ldots, p^*_m\}$, where $\displaystyle p_k^* = \lim_{n \rightarrow \infty} f^n(p_k)$.  For convenience, we will use the notation $p_k^n  = f^n(p_k)$ in what follows.

Let $\epsilon > 0$.  By Lemma \ref{lem:conv} there is an $N$ such that for all $n\geq N$,  
\begin{equation} \label{eq:ineq1}
L(C^{n})-\epsilon < L(C^{n+1}) \leq L(C^{n}).
\end{equation}
Consider the distance $||p^{n+1}_{k} - p^{n+1}_{k-1}||$.  This distance must satisify the inequality
\begin{equation} \label{eq:ineq2}
\frac{L(C^n)}{m}-\epsilon< ||p^{n+1}_{k} - p^{n+1}_{k-1}|| \leq \frac{L(C^n)}{m}
\end{equation}
for $n\geq N$ and $k =1, \ldots, m$.  The right side of this inequality is by construction,
since the points $p^{n+1}_k$ are chosen to be evenly spaced by arclength distance $\frac{L(C^n)}{m}$ along $C^n$.  For the left side, assume that there is some $k^*$ such that $ ||p^{n+1}_{k^*} - p^{n+1}_{k^*-1}|| \leq \frac{L(C^n)}{m}-\epsilon$.  Then
\[
\begin{split}
L(C^{n+1}) &= \sum_{k=1}^m ||p^{n+1}_{k}-p^{n+1}_{k-1} || \\
&= ||p^{n+1}_{k^*} - p^{n+1}_{k^*-1}|| + \sum_{k \neq k^*}||p^{n+1}_{k}-p^n_{k-1} ||\\
&\leq \frac{L(C^n)}{m}-\epsilon + (m-1)\frac{L(C^n)}{m}\\
&< L(C^{n+1}),
\end{split}
\]
where the last two inequalities follow from \eqref{eq:ineq1}.  This is a contradiction, and \eqref{eq:ineq2} follows.

Next, for $n\geq N$ and $k=1, \ldots, m$, consider the distance between points on consecutive iterates, $||p^{n+2}_k- p^{n+1}_k||$.  By construction, the point $p^{n+2}_k$ is placed on $C^{n+1}$ an arclength distance of $\frac{k L(C^{n+1})}{m}$ from $p_0^{n+1}$ along $C^{n+1}$.  By definition, the point $p^{n+1}_k$ lies an arclength distance of $\sum_{i=1}^k || p^{n+1}_i - p^{n+1}_{i-1} ||$ from $p_0^{n+1}$ along $C^{n+1}$.  Thus, the Euclidean distance $||p^{n+2}_k- p^{n+1}_k||$ is bounded by the difference of these arclength distances along $C^{n+1}$:

\begin{equation}
    ||p^{n+2}_k- p^{n+1}_k|| \leq \bigg| \frac{k L(C^{n+1})}{m} - \sum_{i=1}^k || p^{n+1}_i - p^{n+1}_{i-1} || \bigg|.
\end{equation}
With some rearrangment we find
\begin{equation} \label{eq:ineq3}
\begin{split}
\bigg| \frac{k L(C^{n+1})}{m} - \sum_{i=1}^k || p^{n+1}_i - p^{n+1}_{i-1} || \bigg| & = \bigg| \sum_{i=1}^k  \frac{ L(C^{n+1})}{m} - || p^{n+1}_i - p^{n+1}_{i-1} || \bigg| \\
& \leq \sum_{i=1}^k  \bigg|\frac{ L(C^{n+1})}{m} - || p^{n+1}_i - p^{n+1}_{i-1} || \bigg|.
\end{split}
\end{equation}
Now, by \eqref{eq:ineq1},
\[
\frac{L(C^n)}{m} - \epsilon  < \frac{L(C^{n+1})}{m} \leq \frac{L(C^n)}{m},
\]
which, combined with \eqref{eq:ineq2}, yields
\[
\bigg|\frac{ L(C^{n+1})}{m} - || p^{n+1}_i - p^{n+1}_{i-1} || \bigg|< \epsilon
\]
for $i = 1, \ldots, k.$ Thus
\begin{equation} \label{eq:ineqfinal}
\begin{split}
 ||p^{n+2}_k- p^{n+1}_k|| &\leq \sum_{i=1}^k  \bigg|\frac{ L(C^{n+1})}{m} - || p^{n+1}_i - p^{n+1}_{i-1} || \bigg| \\
 & <  k \epsilon.
\end{split}
\end{equation}
Since $m$ is fixed in the respacing process, this provides a uniform bound  $||p^{n+2}_k- p^{n+1}_k|| < m \epsilon$ for $k = 1, \ldots, m $ and $n \geq N$, and thus $\displaystyle \lim_{n\rightarrow \infty} p_k^n$ converges for $k = 0, \ldots, m$.  (The $k=0$ case is immediate since the initial point is always fixed.)  As before, call the limit $p_k^*$ and the limiting curve $C^* = \{p_0^*, \ldots, p_m^*\}$.

Finally, we show that $C^*$ is equilateral.  From \eqref{eq:ineq3} and \eqref{eq:ineqfinal} we see that, for $k = 1, \ldots, m$,
\[
\lim_{n\rightarrow \infty}
\bigg| \frac{k L(C^{n+1})}{m} - \sum_{i=1}^k || p^{n+1}_i - p^{n+1}_{i-1} || \bigg|  = 0.
\]
Passing the limit, we have
\[
\bigg| \frac{k L(C^*)}{m} - \sum_{i=1}^k || p^*_i - p^*_{i-1} || \bigg| = 0,
\]
and thus by Lemma \ref{lem:prebigtheorem}, $C^*$ is equilateral.

\end{proof}

\section{Examples}  \label{sec:examples}

Any given polygonal curve will limit to an equilateral curve, however it appears to be difficult in general to determine the limiting equilateral curve.  In this section we look at some examples where we know the limiting curve or can determine information about it.  We will continue to use the notation $p_k^n  = f^n(p_k)$ and $\displaystyle p_k^* = \lim_{n\rightarrow \infty} p_k^n$ from the proof of Theorem \ref{thm:main}.

\begin{eg}  There are polygonal curves which become equilateral precisely at iteration $n$.  Consider three colinear vertices $p_0, p_1, p_2$, with $p_2$ on the line segment connecting $p_0$ and $p_1$.  Let $d =||p_2-p_0||$. After each iteration, $p_0$ and $p_2$ remain fixed, and $p_1$ moves a distance $d/2$ closer to $p_0$. If $L(C) = n\, d$, then $C$ will become equilateral at precisely iteration $n$.  This example is illustrated in Figure \ref{fig:eq at n}.

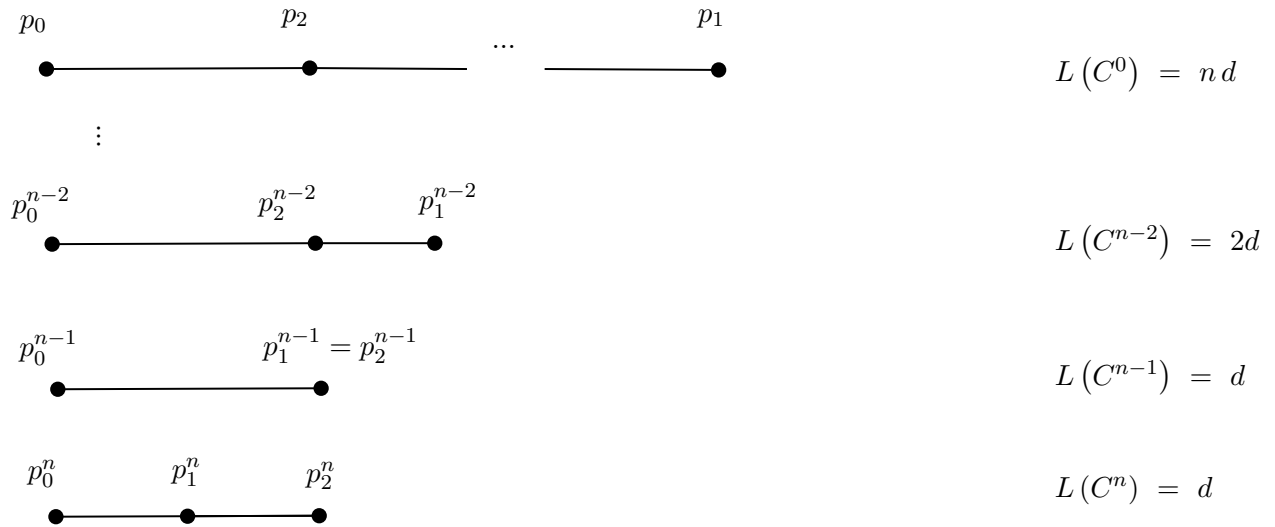
\begin{figure}[ht]
    \centering

\tikzset{every picture/.style={line width=0.75pt}} 

\begin{tikzpicture}[x=0.75pt,y=0.75pt,yscale=-1,xscale=1]

\draw    (35.84,137.43) -- (168.63,136.95) ;
\draw [shift={(168.63,136.95)}, rotate = 359.79] [color={rgb, 255:red, 0; green, 0; blue, 0 }  ][fill={rgb, 255:red, 0; green, 0; blue, 0 }  ][line width=0.75]      (0, 0) circle [x radius= 3.35, y radius= 3.35]   ;
\draw [shift={(35.84,137.43)}, rotate = 359.79] [color={rgb, 255:red, 0; green, 0; blue, 0 }  ][fill={rgb, 255:red, 0; green, 0; blue, 0 }  ][line width=0.75]      (0, 0) circle [x radius= 3.35, y radius= 3.35]   ;
\draw    (168.63,136.95) -- (228.82,136.95) ;
\draw [shift={(228.82,136.95)}, rotate = 0] [color={rgb, 255:red, 0; green, 0; blue, 0 }  ][fill={rgb, 255:red, 0; green, 0; blue, 0 }  ][line width=0.75]      (0, 0) circle [x radius= 3.35, y radius= 3.35]   ;
\draw    (38.7,210.67) -- (149,210.26) -- (171.5,210.18) ;
\draw [shift={(171.5,210.18)}, rotate = 359.79] [color={rgb, 255:red, 0; green, 0; blue, 0 }  ][fill={rgb, 255:red, 0; green, 0; blue, 0 }  ][line width=0.75]      (0, 0) circle [x radius= 3.35, y radius= 3.35]   ;
\draw [shift={(38.7,210.67)}, rotate = 359.79] [color={rgb, 255:red, 0; green, 0; blue, 0 }  ][fill={rgb, 255:red, 0; green, 0; blue, 0 }  ][line width=0.75]      (0, 0) circle [x radius= 3.35, y radius= 3.35]   ;
\draw    (37.75,275) -- (170.54,274.51) ;
\draw [shift={(170.54,274.51)}, rotate = 359.79] [color={rgb, 255:red, 0; green, 0; blue, 0 }  ][fill={rgb, 255:red, 0; green, 0; blue, 0 }  ][line width=0.75]      (0, 0) circle [x radius= 3.35, y radius= 3.35]   ;
\draw [shift={(37.75,275)}, rotate = 359.79] [color={rgb, 255:red, 0; green, 0; blue, 0 }  ][fill={rgb, 255:red, 0; green, 0; blue, 0 }  ][line width=0.75]      (0, 0) circle [x radius= 3.35, y radius= 3.35]   ;
\draw    (104.15,274.76) -- (170.54,274.51) ;
\draw [shift={(104.15,274.76)}, rotate = 359.79] [color={rgb, 255:red, 0; green, 0; blue, 0 }  ][fill={rgb, 255:red, 0; green, 0; blue, 0 }  ][line width=0.75]      (0, 0) circle [x radius= 3.35, y radius= 3.35]   ;
\draw    (32.97,49.04) -- (165.77,48.55) ;
\draw [shift={(165.77,48.55)}, rotate = 359.79] [color={rgb, 255:red, 0; green, 0; blue, 0 }  ][fill={rgb, 255:red, 0; green, 0; blue, 0 }  ][line width=0.75]      (0, 0) circle [x radius= 3.35, y radius= 3.35]   ;
\draw [shift={(32.97,49.04)}, rotate = 359.79] [color={rgb, 255:red, 0; green, 0; blue, 0 }  ][fill={rgb, 255:red, 0; green, 0; blue, 0 }  ][line width=0.75]      (0, 0) circle [x radius= 3.35, y radius= 3.35]   ;
\draw    (284.23,49.04) -- (372,49.5) ;
\draw [shift={(372,49.5)}, rotate = 0.3] [color={rgb, 255:red, 0; green, 0; blue, 0 }  ][fill={rgb, 255:red, 0; green, 0; blue, 0 }  ][line width=0.75]      (0, 0) circle [x radius= 3.35, y radius= 3.35]   ;
\draw    (165.77,48.55) -- (245.06,49.04) ;

\draw (14.11,107.37) node [anchor=north west][inner sep=0.75pt]    {$p_{0}^{n-2}$};
\draw (17.93,179.55) node [anchor=north west][inner sep=0.75pt]    {$p_{0}^{n-1}$};
\draw (22.04,244.83) node [anchor=north west][inner sep=0.75pt]    {$p_{0}^{n}$};
\draw (219.42,105.42) node [anchor=north west][inner sep=0.75pt]    {$p_{1}^{n-2}$};
\draw (138.5,106.37) node [anchor=north west][inner sep=0.75pt]    {$ p_{2}^{n-2}$};
\draw (140.76,178.65) node [anchor=north west][inner sep=0.75pt]    {$p_{1}^{n-1} = p_{2}^{n-1}$};
\draw (162.47,245.88) node [anchor=north west][inner sep=0.75pt]    {$p_{2}^{n}$};
\draw (94.65,243.96) node [anchor=north west][inner sep=0.75pt]    {$p_{1}^{n}$};
\draw (61.63,75.2) node [anchor=north west][inner sep=0.75pt]  [rotate=-90]  {$...$};
\draw (256.12,35.59) node [anchor=north west][inner sep=0.75pt]    {$...$};
\draw (540,41.4) node [anchor=north west][inner sep=0.75pt]    {$L\left( C^{0}\right) \ =\ n\, d$};
\draw (540,252.4) node [anchor=north west][inner sep=0.75pt]    {$L\left( C^{n}\right) \ =\ d$};
\draw (540,194.4) node [anchor=north west][inner sep=0.75pt]    {$L\left( C^{n-1}\right) \ =\ d$};
\draw (540,126.4) node [anchor=north west][inner sep=0.75pt]    {$L\left( C^{n-2}\right) \ =\ 2d$};
\draw (18.11,20.37) node [anchor=north west][inner sep=0.75pt]    {$p_{0}$};
\draw (360,17.37) node [anchor=north west][inner sep=0.75pt]    {$p_{1}$};
\draw (150.11,17.37) node [anchor=north west][inner sep=0.75pt]    {$p_{2}$};

\end{tikzpicture}

    \caption{A polygonal curve that becomes equilateral after $n$ iterations.}
    \label{fig:eq at n}
\end{figure}
\end{eg}

It is not necessary to have colinear vertices to produce an example which becomes equilateral after finitely many iterations. Figure \ref{fig:eveneq} shows a curve which becomes equilateral after exactly two iterations. It is an open question if it is possible to generalize an approach like the one shown in Figure \ref{fig:eveneq} to work for an arbitrary iteration $n$.

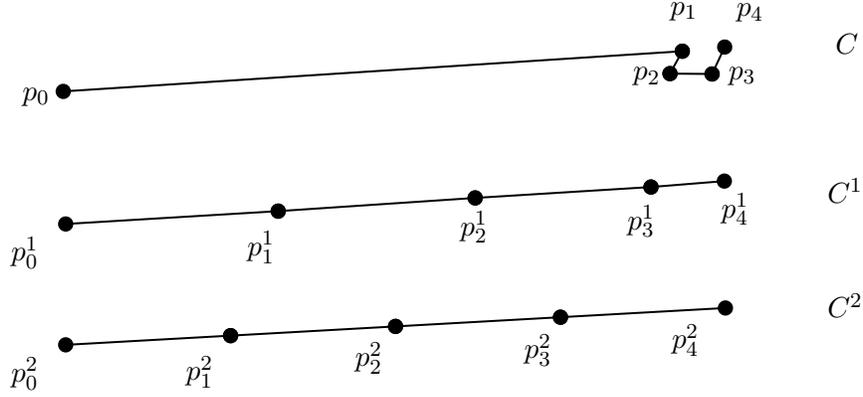
\begin{figure}[ht]
\centering
\tikzset{every picture/.style={line width=0.75pt}} 

\begin{tikzpicture}[x=0.75pt,y=0.75pt,yscale=-1,xscale=1]

\draw [color={rgb, 255:red, 0; green, 0; blue, 0 }  ,draw opacity=1 ][line width=0.75]    (432.8,89.45) -- (120.59,109.73) ;
\draw [shift={(120.59,109.73)}, rotate = 176.28] [color={rgb, 255:red, 0; green, 0; blue, 0 }  ,draw opacity=1 ][fill={rgb, 255:red, 0; green, 0; blue, 0 }  ,fill opacity=1 ][line width=0.75]      (0, 0) circle [x radius= 3.35, y radius= 3.35]   ;
\draw [shift={(432.8,89.45)}, rotate = 176.28] [color={rgb, 255:red, 0; green, 0; blue, 0 }  ,draw opacity=1 ][fill={rgb, 255:red, 0; green, 0; blue, 0 }  ,fill opacity=1 ][line width=0.75]      (0, 0) circle [x radius= 3.35, y radius= 3.35]   ;
\draw [color={rgb, 255:red, 0; green, 0; blue, 0 }  ,draw opacity=1 ][line width=0.75]    (454.18,87.34) -- (447.82,100.94) ;
\draw [shift={(447.82,100.94)}, rotate = 115.06] [color={rgb, 255:red, 0; green, 0; blue, 0 }  ,draw opacity=1 ][fill={rgb, 255:red, 0; green, 0; blue, 0 }  ,fill opacity=1 ][line width=0.75]      (0, 0) circle [x radius= 3.35, y radius= 3.35]   ;
\draw [shift={(454.18,87.34)}, rotate = 115.06] [color={rgb, 255:red, 0; green, 0; blue, 0 }  ,draw opacity=1 ][fill={rgb, 255:red, 0; green, 0; blue, 0 }  ,fill opacity=1 ][line width=0.75]      (0, 0) circle [x radius= 3.35, y radius= 3.35]   ;
\draw [color={rgb, 255:red, 0; green, 0; blue, 0 }  ,draw opacity=1 ][line width=0.75]    (447.82,100.94) -- (426.59,100.75) ;
\draw [shift={(426.59,100.75)}, rotate = 180.52] [color={rgb, 255:red, 0; green, 0; blue, 0 }  ,draw opacity=1 ][fill={rgb, 255:red, 0; green, 0; blue, 0 }  ,fill opacity=1 ][line width=0.75]      (0, 0) circle [x radius= 3.35, y radius= 3.35]   ;
\draw [shift={(447.82,100.94)}, rotate = 180.52] [color={rgb, 255:red, 0; green, 0; blue, 0 }  ,draw opacity=1 ][fill={rgb, 255:red, 0; green, 0; blue, 0 }  ,fill opacity=1 ][line width=0.75]      (0, 0) circle [x radius= 3.35, y radius= 3.35]   ;
\draw [color={rgb, 255:red, 0; green, 0; blue, 0 }  ,draw opacity=1 ][line width=0.75]    (426.59,100.75) -- (432.8,89.45) ;
\draw [shift={(432.8,89.45)}, rotate = 298.79] [color={rgb, 255:red, 0; green, 0; blue, 0 }  ,draw opacity=1 ][fill={rgb, 255:red, 0; green, 0; blue, 0 }  ,fill opacity=1 ][line width=0.75]      (0, 0) circle [x radius= 3.35, y radius= 3.35]   ;
\draw [shift={(426.59,100.75)}, rotate = 298.79] [color={rgb, 255:red, 0; green, 0; blue, 0 }  ,draw opacity=1 ][fill={rgb, 255:red, 0; green, 0; blue, 0 }  ,fill opacity=1 ][line width=0.75]      (0, 0) circle [x radius= 3.35, y radius= 3.35]   ;
\draw [color={rgb, 255:red, 0; green, 0; blue, 0 }  ,draw opacity=1 ][line width=0.75]    (229,170.17) -- (121.82,176.65) ;
\draw [shift={(121.82,176.65)}, rotate = 176.54] [color={rgb, 255:red, 0; green, 0; blue, 0 }  ,draw opacity=1 ][fill={rgb, 255:red, 0; green, 0; blue, 0 }  ,fill opacity=1 ][line width=0.75]      (0, 0) circle [x radius= 3.35, y radius= 3.35]   ;
\draw [shift={(229,170.17)}, rotate = 176.54] [color={rgb, 255:red, 0; green, 0; blue, 0 }  ,draw opacity=1 ][fill={rgb, 255:red, 0; green, 0; blue, 0 }  ,fill opacity=1 ][line width=0.75]      (0, 0) circle [x radius= 3.35, y radius= 3.35]   ;
\draw [color={rgb, 255:red, 0; green, 0; blue, 0 }  ,draw opacity=1 ][line width=0.75]    (454,155) -- (417,158) ;
\draw [shift={(417,158)}, rotate = 175.36] [color={rgb, 255:red, 0; green, 0; blue, 0 }  ,draw opacity=1 ][fill={rgb, 255:red, 0; green, 0; blue, 0 }  ,fill opacity=1 ][line width=0.75]      (0, 0) circle [x radius= 3.35, y radius= 3.35]   ;
\draw [shift={(454,155)}, rotate = 175.36] [color={rgb, 255:red, 0; green, 0; blue, 0 }  ,draw opacity=1 ][fill={rgb, 255:red, 0; green, 0; blue, 0 }  ,fill opacity=1 ][line width=0.75]      (0, 0) circle [x radius= 3.35, y radius= 3.35]   ;
\draw [color={rgb, 255:red, 0; green, 0; blue, 0 }  ,draw opacity=1 ][line width=0.75]    (328.33,163.5) -- (229,170.17) ;
\draw [shift={(229,170.17)}, rotate = 176.16] [color={rgb, 255:red, 0; green, 0; blue, 0 }  ,draw opacity=1 ][fill={rgb, 255:red, 0; green, 0; blue, 0 }  ,fill opacity=1 ][line width=0.75]      (0, 0) circle [x radius= 3.35, y radius= 3.35]   ;
\draw [shift={(328.33,163.5)}, rotate = 176.16] [color={rgb, 255:red, 0; green, 0; blue, 0 }  ,draw opacity=1 ][fill={rgb, 255:red, 0; green, 0; blue, 0 }  ,fill opacity=1 ][line width=0.75]      (0, 0) circle [x radius= 3.35, y radius= 3.35]   ;
\draw [color={rgb, 255:red, 0; green, 0; blue, 0 }  ,draw opacity=1 ][line width=0.75]    (417,158) -- (328.33,163.5) ;
\draw [shift={(328.33,163.5)}, rotate = 176.45] [color={rgb, 255:red, 0; green, 0; blue, 0 }  ,draw opacity=1 ][fill={rgb, 255:red, 0; green, 0; blue, 0 }  ,fill opacity=1 ][line width=0.75]      (0, 0) circle [x radius= 3.35, y radius= 3.35]   ;
\draw [shift={(417,158)}, rotate = 176.45] [color={rgb, 255:red, 0; green, 0; blue, 0 }  ,draw opacity=1 ][fill={rgb, 255:red, 0; green, 0; blue, 0 }  ,fill opacity=1 ][line width=0.75]      (0, 0) circle [x radius= 3.35, y radius= 3.35]   ;
\draw [color={rgb, 255:red, 0; green, 0; blue, 0 }  ,draw opacity=1 ][line width=0.75]    (205,233) -- (121.82,237.65) ;
\draw [shift={(121.82,237.65)}, rotate = 176.8] [color={rgb, 255:red, 0; green, 0; blue, 0 }  ,draw opacity=1 ][fill={rgb, 255:red, 0; green, 0; blue, 0 }  ,fill opacity=1 ][line width=0.75]      (0, 0) circle [x radius= 3.35, y radius= 3.35]   ;
\draw [shift={(205,233)}, rotate = 176.8] [color={rgb, 255:red, 0; green, 0; blue, 0 }  ,draw opacity=1 ][fill={rgb, 255:red, 0; green, 0; blue, 0 }  ,fill opacity=1 ][line width=0.75]      (0, 0) circle [x radius= 3.35, y radius= 3.35]   ;
\draw [color={rgb, 255:red, 0; green, 0; blue, 0 }  ,draw opacity=1 ][line width=0.75]    (454.55,219.05) -- (371.37,223.7) ;
\draw [shift={(371.37,223.7)}, rotate = 176.8] [color={rgb, 255:red, 0; green, 0; blue, 0 }  ,draw opacity=1 ][fill={rgb, 255:red, 0; green, 0; blue, 0 }  ,fill opacity=1 ][line width=0.75]      (0, 0) circle [x radius= 3.35, y radius= 3.35]   ;
\draw [shift={(454.55,219.05)}, rotate = 176.8] [color={rgb, 255:red, 0; green, 0; blue, 0 }  ,draw opacity=1 ][fill={rgb, 255:red, 0; green, 0; blue, 0 }  ,fill opacity=1 ][line width=0.75]      (0, 0) circle [x radius= 3.35, y radius= 3.35]   ;
\draw [color={rgb, 255:red, 0; green, 0; blue, 0 }  ,draw opacity=1 ][line width=0.75]    (371.37,223.7) -- (288.18,228.35) ;
\draw [shift={(288.18,228.35)}, rotate = 176.8] [color={rgb, 255:red, 0; green, 0; blue, 0 }  ,draw opacity=1 ][fill={rgb, 255:red, 0; green, 0; blue, 0 }  ,fill opacity=1 ][line width=0.75]      (0, 0) circle [x radius= 3.35, y radius= 3.35]   ;
\draw [shift={(371.37,223.7)}, rotate = 176.8] [color={rgb, 255:red, 0; green, 0; blue, 0 }  ,draw opacity=1 ][fill={rgb, 255:red, 0; green, 0; blue, 0 }  ,fill opacity=1 ][line width=0.75]      (0, 0) circle [x radius= 3.35, y radius= 3.35]   ;

\draw [color={rgb, 255:red, 0; green, 0; blue, 0 }  ,draw opacity=1 ][line width=0.75]    (288.18,228.35) -- (205,233) ;
\draw [shift={(205,233)}, rotate = 176.8] [color={rgb, 255:red, 0; green, 0; blue, 0 }  ,draw opacity=1 ][fill={rgb, 255:red, 0; green, 0; blue, 0 }  ,fill opacity=1 ][line width=0.75]      (0, 0) circle [x radius= 3.35, y radius= 3.35]   ;
\draw [shift={(288.18,228.35)}, rotate = 176.8] [color={rgb, 255:red, 0; green, 0; blue, 0 }  ,draw opacity=1 ][fill={rgb, 255:red, 0; green, 0; blue, 0 }  ,fill opacity=1 ][line width=0.75]      (0, 0) circle [x radius= 3.35, y radius= 3.35]   ;

\draw (98.52,106.2) node [anchor=north west][inner sep=0.75pt]    {$p_{0}$};
\draw (425.24,63.04) node [anchor=north west][inner sep=0.75pt]    {$p_{1}$};
\draw (406.6,96.08) node [anchor=north west][inner sep=0.75pt]    {$p_{2}$};
\draw (458.66,63.97) node [anchor=north west][inner sep=0.75pt]    {$p_{4}$};
\draw (454.73,95.67) node [anchor=north west][inner sep=0.75pt]    {$p_{3}$};
\draw (92.84,181.51) node [anchor=north west][inner sep=0.75pt]    {$p_{0}^1$};
\draw (212.06,178.74) node [anchor=north west][inner sep=0.75pt]    {$p_{1}^1$};
\draw (319.37,168.62) node [anchor=north west][inner sep=0.75pt]    {$p_{2}^1$};
\draw (451.03,159.62) node [anchor=north west][inner sep=0.75pt]    {$p_{4}^1$};
\draw (403.7,165.62) node [anchor=north west][inner sep=0.75pt]    {$p_{3}^1$};
\draw (92.84,242.51) node [anchor=north west][inner sep=0.75pt]    {$p_{0}^2$};
\draw (181.06,241.74) node [anchor=north west][inner sep=0.75pt]    {$p_{1}^2$};
\draw (266.37,234.62) node [anchor=north west][inner sep=0.75pt]    {$p_{2}^2$};
\draw (426.03,225.62) node [anchor=north west][inner sep=0.75pt]    {$p_{4}^2$};
\draw (351.7,230.62) node [anchor=north west][inner sep=0.75pt]    {$p_{3}^2$};
\draw (508.52,79.2) node [anchor=north west][inner sep=0.75pt]    {$C$};
\draw (504.52,210.2) node [anchor=north west][inner sep=0.75pt]    {$C^{2}$};
\draw (504.52,152.2) node [anchor=north west][inner sep=0.75pt]    {$C^{1}$};

\end{tikzpicture}

    \caption{A polygonal curve that becomes equilateral after two iterations.}
    \label{fig:eveneq}

\end{figure}

\begin{eg}  We next consider triangles, viewed as closed polygonal curves with four vertices $\{p_0,p_1,p_2,p_3=p_0\}$.  By Theorem \ref{thm:main}, the only possible limiting polygonal curves are a single point, or an equilateral triangle.  We consider starting curves which will realize either of these possibilities.

Consider an isosceles triangle $C$ with angle $\theta >  \frac{\pi}{3}$ at $p_0$, as shown in Figure \ref{fig:13}.  The points $p_1$ and $p_2$ will be mapped symmetrically under iteration to points on the side opposite to $p_0$.  The angle $\theta^n$ at $p_0^n$ will approach $\pi/3$ and $C$ will approach an equilateral triangle.

\begin{figure}[ht]
    \centering
 \tikzset{every picture/.style={line width=0.75pt}} 

\begin{tikzpicture}[x=0.75pt,y=0.75pt,yscale=-1,xscale=1]

\draw    (315.42,222.34) -- (99.04,85.98) ;
\draw [shift={(99.04,85.98)}, rotate = 212.22] [color={rgb, 255:red, 0; green, 0; blue, 0 }  ][fill={rgb, 255:red, 0; green, 0; blue, 0 }  ][line width=0.75]      (0, 0) circle [x radius= 3.35, y radius= 3.35]   ;
\draw [shift={(315.42,222.34)}, rotate = 212.22] [color={rgb, 255:red, 0; green, 0; blue, 0 }  ][fill={rgb, 255:red, 0; green, 0; blue, 0 }  ][line width=0.75]      (0, 0) circle [x radius= 3.35, y radius= 3.35]   ;
\draw    (535.86,86.61) -- (99.04,85.98) ;
\draw [shift={(99.04,85.98)}, rotate = 180.08] [color={rgb, 255:red, 0; green, 0; blue, 0 }  ][fill={rgb, 255:red, 0; green, 0; blue, 0 }  ][line width=0.75]      (0, 0) circle [x radius= 3.35, y radius= 3.35]   ;
\draw [shift={(535.86,86.61)}, rotate = 180.08] [color={rgb, 255:red, 0; green, 0; blue, 0 }  ][fill={rgb, 255:red, 0; green, 0; blue, 0 }  ][line width=0.75]      (0, 0) circle [x radius= 3.35, y radius= 3.35]   ;
\draw    (315.42,222.34) -- (535.86,86.61) ;
\draw [shift={(535.86,86.61)}, rotate = 328.38] [color={rgb, 255:red, 0; green, 0; blue, 0 }  ][fill={rgb, 255:red, 0; green, 0; blue, 0 }  ][line width=0.75]      (0, 0) circle [x radius= 3.35, y radius= 3.35]   ;
\draw [shift={(315.42,222.34)}, rotate = 328.38] [color={rgb, 255:red, 0; green, 0; blue, 0 }  ][fill={rgb, 255:red, 0; green, 0; blue, 0 }  ][line width=0.75]      (0, 0) circle [x radius= 3.35, y radius= 3.35]   ;
\draw    (101.39,67.66) -- (181.5,67.99) ;
\draw [shift={(184.5,68)}, rotate = 180.24] [fill={rgb, 255:red, 0; green, 0; blue, 0 }  ][line width=0.08]  [draw opacity=0] (10.72,-5.15) -- (0,0) -- (10.72,5.15) -- (7.12,0) -- cycle    ;
\draw [color={rgb, 255:red, 0; green, 0; blue, 0 }  ,draw opacity=0.4 ][line width=1.5]  [dash pattern={on 5.63pt off 4.5pt}]  (315.42,222.34) -- (401.01,86.78) ;
\draw [shift={(401.01,86.78)}, rotate = 302.27] [color={rgb, 255:red, 0; green, 0; blue, 0 }  ,draw opacity=0.4 ][fill={rgb, 255:red, 0; green, 0; blue, 0 }  ,fill opacity=0.4 ][line width=1.5]      (0, 0) circle [x radius= 4.36, y radius= 4.36]   ;
\draw [shift={(315.42,222.34)}, rotate = 302.27] [color={rgb, 255:red, 0; green, 0; blue, 0 }  ,draw opacity=0.4 ][fill={rgb, 255:red, 0; green, 0; blue, 0 }  ,fill opacity=0.4 ][line width=1.5]      (0, 0) circle [x radius= 4.36, y radius= 4.36]   ;
\draw [color={rgb, 255:red, 0; green, 0; blue, 0 }  ,draw opacity=0.4 ][line width=1.5]  [dash pattern={on 5.63pt off 4.5pt}]  (236.38,86.78) -- (315.42,222.34) ;
\draw [shift={(315.42,222.34)}, rotate = 59.75] [color={rgb, 255:red, 0; green, 0; blue, 0 }  ,draw opacity=0.4 ][fill={rgb, 255:red, 0; green, 0; blue, 0 }  ,fill opacity=0.4 ][line width=1.5]      (0, 0) circle [x radius= 4.36, y radius= 4.36]   ;
\draw [shift={(236.38,86.78)}, rotate = 59.75] [color={rgb, 255:red, 0; green, 0; blue, 0 }  ,draw opacity=0.4 ][fill={rgb, 255:red, 0; green, 0; blue, 0 }  ,fill opacity=0.4 ][line width=1.5]      (0, 0) circle [x radius= 4.36, y radius= 4.36]   ;
\draw    (532.21,66.66) -- (456.5,66.02) ;
\draw [shift={(453.5,66)}, rotate = 360.48] [fill={rgb, 255:red, 0; green, 0; blue, 0 }  ][line width=0.08]  [draw opacity=0] (10.72,-5.15) -- (0,0) -- (10.72,5.15) -- (7.12,0) -- cycle    ;
\draw [color={rgb, 255:red, 0; green, 0; blue, 0 }  ,draw opacity=0.4 ]   (158,86.33) -- (315.95,222.59) ;
\draw [color={rgb, 255:red, 0; green, 0; blue, 0 }  ,draw opacity=0.3 ]   (198.33,86.33) -- (315.95,222.59) ;
\draw [color={rgb, 255:red, 0; green, 0; blue, 0 }  ,draw opacity=0.2 ][fill={rgb, 255:red, 0; green, 0; blue, 0 }  ,fill opacity=0.2 ]   (218.33,86.67) -- (315.95,222.59) ;
\draw [color={rgb, 255:red, 0; green, 0; blue, 0 }  ,draw opacity=0.1 ][fill={rgb, 255:red, 0; green, 0; blue, 0 }  ,fill opacity=0.2 ]   (228.67,86.67) -- (315.42,222.34) ;
\draw [color={rgb, 255:red, 0; green, 0; blue, 0 }  ,draw opacity=0.4 ]   (476.5,86) -- (316.28,222.26) ;
\draw [color={rgb, 255:red, 0; green, 0; blue, 0 }  ,draw opacity=0.3 ]   (435.59,86) -- (316.28,222.26) ;
\draw [color={rgb, 255:red, 0; green, 0; blue, 0 }  ,draw opacity=0.2 ][fill={rgb, 255:red, 0; green, 0; blue, 0 }  ,fill opacity=0.2 ]   (415.3,86.33) -- (316.28,222.26) ;
\draw [color={rgb, 255:red, 0; green, 0; blue, 0 }  ,draw opacity=0.1 ][fill={rgb, 255:red, 0; green, 0; blue, 0 }  ,fill opacity=0.2 ]   (404.82,86.33) -- (316.81,222.01) ;
\draw [color={rgb, 255:red, 0; green, 0; blue, 0 }  ,draw opacity=0.4 ][line width=1.5]  [dash pattern={on 5.63pt off 4.5pt}]  (236.38,86.78) -- (401.01,86.78) ;
\draw [shift={(401.01,86.78)}, rotate = 0] [color={rgb, 255:red, 0; green, 0; blue, 0 }  ,draw opacity=0.4 ][fill={rgb, 255:red, 0; green, 0; blue, 0 }  ,fill opacity=0.4 ][line width=1.5]      (0, 0) circle [x radius= 4.36, y radius= 4.36]   ;
\draw [shift={(236.38,86.78)}, rotate = 0] [color={rgb, 255:red, 0; green, 0; blue, 0 }  ,draw opacity=0.4 ][fill={rgb, 255:red, 0; green, 0; blue, 0 }  ,fill opacity=0.4 ][line width=1.5]      (0, 0) circle [x radius= 4.36, y radius= 4.36]   ;

\draw (308.84,229.51) node [anchor=north west][inner sep=0.75pt]    {$p_{0}$};
\draw (82.69,58.05) node [anchor=north west][inner sep=0.75pt]    {$p_{1}$};
\draw (539.03,55.28) node [anchor=north west][inner sep=0.75pt]    {$p_{2}$};
\draw (219.73,52.74) node [anchor=north west][inner sep=0.75pt]    {$p_{1}^*$};
\draw (388.73,53.74) node [anchor=north west][inner sep=0.75pt]    {$p_{2}^*$};

\end{tikzpicture}

    \caption{Arclength respacing iteration for an isosceles triangle with angle $\theta > \frac{\pi}{3}$ at $p_0$ converges to an equilateral triangle.}
    \label{fig:13}
\end{figure}
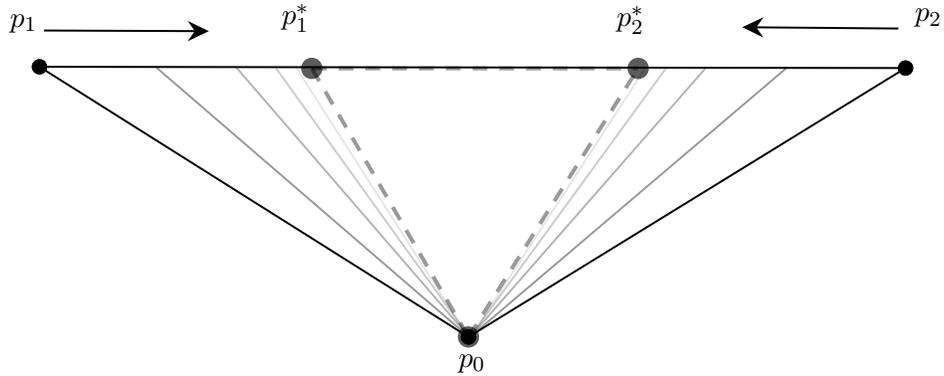

Next consider a triangle $C$ with an angle $\theta < \frac{\pi}{3}$ at the starting vertex $p_0$.  By construction, this angle does not increase under iteration, so $\theta^{n+1} \leq \theta^{n}$, where $\theta^n$ is the angle at the vertex $p_0^n$ of $C^n$.  Thus $C^n$ cannot approach an equilateral triangle, and must converge to a point.  A special case of this is the iteration of an isosceles triangle with angle $\theta < \frac{\pi}{3}$ at $p_0$, which will produce a sequence of similar triangles shrinking to a point, as shown in Figure \ref{fig:12}.

\begin{figure}[ht]
    \centering
\tikzset{every picture/.style={line width=0.75pt}} 

\begin{tikzpicture}[x=0.75pt,y=0.75pt,yscale=-1,xscale=1]

\draw    (320.35,257) -- (251.5,50.33) ;
\draw [shift={(251.5,50.33)}, rotate = 251.57] [color={rgb, 255:red, 0; green, 0; blue, 0 }  ][fill={rgb, 255:red, 0; green, 0; blue, 0 }  ][line width=0.75]      (0, 0) circle [x radius= 3.35, y radius= 3.35]   ;
\draw [shift={(320.35,257)}, rotate = 251.57] [color={rgb, 255:red, 0; green, 0; blue, 0 }  ][fill={rgb, 255:red, 0; green, 0; blue, 0 }  ][line width=0.75]      (0, 0) circle [x radius= 3.35, y radius= 3.35]   ;
\draw    (390.5,51.28) -- (251.5,50.33) ;
\draw [shift={(251.5,50.33)}, rotate = 180.39] [color={rgb, 255:red, 0; green, 0; blue, 0 }  ][fill={rgb, 255:red, 0; green, 0; blue, 0 }  ][line width=0.75]      (0, 0) circle [x radius= 3.35, y radius= 3.35]   ;
\draw [shift={(390.5,51.28)}, rotate = 180.39] [color={rgb, 255:red, 0; green, 0; blue, 0 }  ][fill={rgb, 255:red, 0; green, 0; blue, 0 }  ][line width=0.75]      (0, 0) circle [x radius= 3.35, y radius= 3.35]   ;
\draw    (320.35,257) -- (390.5,51.28) ;
\draw [shift={(390.5,51.28)}, rotate = 288.83] [color={rgb, 255:red, 0; green, 0; blue, 0 }  ][fill={rgb, 255:red, 0; green, 0; blue, 0 }  ][line width=0.75]      (0, 0) circle [x radius= 3.35, y radius= 3.35]   ;
\draw [shift={(320.35,257)}, rotate = 288.83] [color={rgb, 255:red, 0; green, 0; blue, 0 }  ][fill={rgb, 255:red, 0; green, 0; blue, 0 }  ][line width=0.75]      (0, 0) circle [x radius= 3.35, y radius= 3.35]   ;
\draw [color={rgb, 255:red, 0; green, 0; blue, 0 }  ,draw opacity=0.6 ]   (367,121.62) -- (274.17,120.83) ;
\draw [shift={(274.17,120.83)}, rotate = 180.49] [color={rgb, 255:red, 0; green, 0; blue, 0 }  ,draw opacity=0.6 ][fill={rgb, 255:red, 0; green, 0; blue, 0 }  ,fill opacity=0.6 ][line width=0.75]      (0, 0) circle [x radius= 3.35, y radius= 3.35]   ;
\draw [shift={(367,121.62)}, rotate = 180.49] [color={rgb, 255:red, 0; green, 0; blue, 0 }  ,draw opacity=0.6 ][fill={rgb, 255:red, 0; green, 0; blue, 0 }  ,fill opacity=0.6 ][line width=0.75]      (0, 0) circle [x radius= 3.35, y radius= 3.35]   ;
\draw [color={rgb, 255:red, 0; green, 0; blue, 0 }  ,draw opacity=0.5 ]   (353,162.72) -- (288.26,162.83) ;
\draw [shift={(288.26,162.83)}, rotate = 179.9] [color={rgb, 255:red, 0; green, 0; blue, 0 }  ,draw opacity=0.5 ][fill={rgb, 255:red, 0; green, 0; blue, 0 }  ,fill opacity=0.5 ][line width=0.75]      (0, 0) circle [x radius= 3.35, y radius= 3.35]   ;
\draw [shift={(353,162.72)}, rotate = 179.9] [color={rgb, 255:red, 0; green, 0; blue, 0 }  ,draw opacity=0.5 ][fill={rgb, 255:red, 0; green, 0; blue, 0 }  ,fill opacity=0.5 ][line width=0.75]      (0, 0) circle [x radius= 3.35, y radius= 3.35]   ;
\draw [color={rgb, 255:red, 0; green, 0; blue, 0 }  ,draw opacity=0.4 ]   (345,186.75) -- (297.52,186.91) ;
\draw [shift={(297.52,186.91)}, rotate = 179.81] [color={rgb, 255:red, 0; green, 0; blue, 0 }  ,draw opacity=0.4 ][fill={rgb, 255:red, 0; green, 0; blue, 0 }  ,fill opacity=0.4 ][line width=0.75]      (0, 0) circle [x radius= 3.35, y radius= 3.35]   ;
\draw [shift={(345,186.75)}, rotate = 179.81] [color={rgb, 255:red, 0; green, 0; blue, 0 }  ,draw opacity=0.4 ][fill={rgb, 255:red, 0; green, 0; blue, 0 }  ,fill opacity=0.4 ][line width=0.75]      (0, 0) circle [x radius= 3.35, y radius= 3.35]   ;
\draw [color={rgb, 255:red, 0; green, 0; blue, 0 }  ,draw opacity=0.3 ]   (338.33,203.82) -- (303.93,203.93) ;
\draw [shift={(303.93,203.93)}, rotate = 179.81] [color={rgb, 255:red, 0; green, 0; blue, 0 }  ,draw opacity=0.3 ][fill={rgb, 255:red, 0; green, 0; blue, 0 }  ,fill opacity=0.3 ][line width=0.75]      (0, 0) circle [x radius= 3.35, y radius= 3.35]   ;
\draw [shift={(338.33,203.82)}, rotate = 179.81] [color={rgb, 255:red, 0; green, 0; blue, 0 }  ,draw opacity=0.3 ][fill={rgb, 255:red, 0; green, 0; blue, 0 }  ,fill opacity=0.3 ][line width=0.75]      (0, 0) circle [x radius= 3.35, y radius= 3.35]   ;
\draw [color={rgb, 255:red, 0; green, 0; blue, 0 }  ,draw opacity=0.15 ]   (328.33,233.54) -- (312.93,233.33) ;
\draw [shift={(312.93,233.33)}, rotate = 180.76] [color={rgb, 255:red, 0; green, 0; blue, 0 }  ,draw opacity=0.15 ][fill={rgb, 255:red, 0; green, 0; blue, 0 }  ,fill opacity=0.15 ][line width=0.75]      (0, 0) circle [x radius= 3.35, y radius= 3.35]   ;
\draw [shift={(328.33,233.54)}, rotate = 180.76] [color={rgb, 255:red, 0; green, 0; blue, 0 }  ,draw opacity=0.15 ][fill={rgb, 255:red, 0; green, 0; blue, 0 }  ,fill opacity=0.15 ][line width=0.75]      (0, 0) circle [x radius= 3.35, y radius= 3.35]   ;
\draw [color={rgb, 255:red, 0; green, 0; blue, 0 }  ,draw opacity=0.2 ]   (333.67,218.99) -- (307.93,219.11) ;
\draw [shift={(307.93,219.11)}, rotate = 179.75] [color={rgb, 255:red, 0; green, 0; blue, 0 }  ,draw opacity=0.2 ][fill={rgb, 255:red, 0; green, 0; blue, 0 }  ,fill opacity=0.2 ][line width=0.75]      (0, 0) circle [x radius= 3.35, y radius= 3.35]   ;
\draw [shift={(333.67,218.99)}, rotate = 179.75] [color={rgb, 255:red, 0; green, 0; blue, 0 }  ,draw opacity=0.2 ][fill={rgb, 255:red, 0; green, 0; blue, 0 }  ,fill opacity=0.2 ][line width=0.75]      (0, 0) circle [x radius= 3.35, y radius= 3.35]   ;
\draw [color={rgb, 255:red, 0; green, 0; blue, 0 }  ,draw opacity=0.1 ]   (326.33,240.49) -- (314.93,240.92) ;
\draw [shift={(314.93,240.92)}, rotate = 177.85] [color={rgb, 255:red, 0; green, 0; blue, 0 }  ,draw opacity=0.1 ][fill={rgb, 255:red, 0; green, 0; blue, 0 }  ,fill opacity=0.1 ][line width=0.75]      (0, 0) circle [x radius= 3.35, y radius= 3.35]   ;
\draw [shift={(326.33,240.49)}, rotate = 177.85] [color={rgb, 255:red, 0; green, 0; blue, 0 }  ,draw opacity=0.1 ][fill={rgb, 255:red, 0; green, 0; blue, 0 }  ,fill opacity=0.1 ][line width=0.75]      (0, 0) circle [x radius= 3.35, y radius= 3.35]   ;
\draw [color={rgb, 255:red, 0; green, 0; blue, 0 }  ,draw opacity=0.05 ]   (323.67,247.45) -- (315.93,247.56) ;
\draw [shift={(315.93,247.56)}, rotate = 179.17] [color={rgb, 255:red, 0; green, 0; blue, 0 }  ,draw opacity=0.05 ][fill={rgb, 255:red, 0; green, 0; blue, 0 }  ,fill opacity=0.05 ][line width=0.75]      (0, 0) circle [x radius= 3.35, y radius= 3.35]   ;
\draw [shift={(323.67,247.45)}, rotate = 179.17] [color={rgb, 255:red, 0; green, 0; blue, 0 }  ,draw opacity=0.05 ][fill={rgb, 255:red, 0; green, 0; blue, 0 }  ,fill opacity=0.05 ][line width=0.75]      (0, 0) circle [x radius= 3.35, y radius= 3.35]   ;
\draw [color={rgb, 255:red, 0; green, 0; blue, 0 }  ,draw opacity=0.02 ]   (322.75,251.28) -- (317.93,251.35) ;
\draw [shift={(317.93,251.35)}, rotate = 179.13] [color={rgb, 255:red, 0; green, 0; blue, 0 }  ,draw opacity=0.02 ][fill={rgb, 255:red, 0; green, 0; blue, 0 }  ,fill opacity=0.02 ][line width=0.75]      (0, 0) circle [x radius= 3.35, y radius= 3.35]   ;
\draw [shift={(322.75,251.28)}, rotate = 179.13] [color={rgb, 255:red, 0; green, 0; blue, 0 }  ,draw opacity=0.02 ][fill={rgb, 255:red, 0; green, 0; blue, 0 }  ,fill opacity=0.02 ][line width=0.75]      (0, 0) circle [x radius= 3.35, y radius= 3.35]   ;
\draw    (243.5,59) -- (262.66,113.1) ;
\draw [shift={(263.67,115.93)}, rotate = 250.49] [fill={rgb, 255:red, 0; green, 0; blue, 0 }  ][line width=0.08]  [draw opacity=0] (10.72,-5.15) -- (0,0) -- (10.72,5.15) -- (7.12,0) -- cycle    ;
\draw    (403,59.29) -- (383.52,113.18) ;
\draw [shift={(382.5,116)}, rotate = 289.87] [fill={rgb, 255:red, 0; green, 0; blue, 0 }  ][line width=0.08]  [draw opacity=0] (10.72,-5.15) -- (0,0) -- (10.72,5.15) -- (7.12,0) -- cycle    ;

\draw (311.5,29.68) node [anchor=north west][inner sep=0.75pt]    {$C$};
\draw (314.96,102.68) node [anchor=north west][inner sep=0.75pt]    {$C^{1}$};
\draw (311.96,142.09) node [anchor=north west][inner sep=0.75pt]    {$C^{2}$};
\draw (313.3,168.17) node [anchor=north west][inner sep=0.75pt]    {$C^{3}$};
\draw (304.63,265.8) node [anchor=north west][inner sep=0.75pt]    {$p_{0}$};
\draw (222.63,41.8) node [anchor=north west][inner sep=0.75pt]    {$p_{1}$};
\draw (397.63,35.8) node [anchor=north west][inner sep=0.75pt]    {$p_{2}$};
\draw (240.63,118.8) node [anchor=north west][inner sep=0.75pt]    {$p_{1}^1$};
\draw (381,117.02) node [anchor=north west][inner sep=0.75pt]    {$p_{2}^1$};

\end{tikzpicture}
    \caption{Arclength respacing iteration for a triangle with angle $\theta < \frac{\pi}{3}$ at $p_0$ converges to a point.}
    \label{fig:12}
\end{figure}
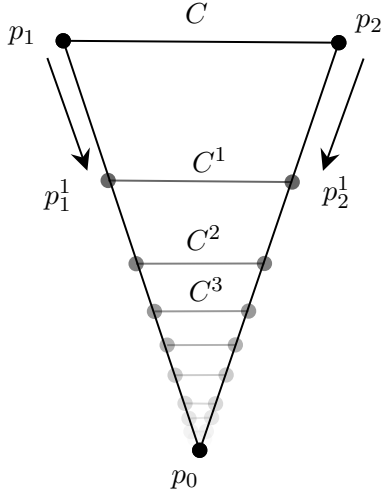

\end{eg}

\begin{eg} We consider quadrilaterals, viewed as closed polygonal curves with five vertices $\{p_0,p_1,p_2,p_3,p_4=p_0\}$.  By Theorem \ref{thm:main}, a quadrilateral must tend toward an equilateral polygonal curve, in this case a rhombus, degenerate rhombus, or a single point.  If the polygonal curve $\{p_0, p_1, p_2\}$ has the same length as the polygonal curve $\{p_2,p_3,p_4\}$, the vertex $p_2$ will remain fixed under iteration.  The limiting curve will be a rhombus.  A special case of this is the iteration of a parallelogram, shown in Figure \ref{fig:parallelogram}.

\begin{figure}[ht]
    \centering
    \tikzset{every picture/.style={line width=0.75pt}} 

\begin{tikzpicture}[x=0.75pt,y=0.75pt,yscale=-1,xscale=1]

\draw [color={rgb, 255:red, 0; green, 0; blue, 0 }  ,draw opacity=1 ][line width=0.75]    (167,200) -- (374.73,203.34) ;
\draw [shift={(374.73,203.34)}, rotate = 0.92] [color={rgb, 255:red, 0; green, 0; blue, 0 }  ,draw opacity=1 ][fill={rgb, 255:red, 0; green, 0; blue, 0 }  ,fill opacity=1 ][line width=0.75]      (0, 0) circle [x radius= 3.35, y radius= 3.35]   ;
\draw [shift={(167,200)}, rotate = 0.92] [color={rgb, 255:red, 0; green, 0; blue, 0 }  ,draw opacity=1 ][fill={rgb, 255:red, 0; green, 0; blue, 0 }  ,fill opacity=1 ][line width=0.75]      (0, 0) circle [x radius= 3.35, y radius= 3.35]   ;
\draw [color={rgb, 255:red, 0; green, 0; blue, 0 }  ,draw opacity=1 ][line width=0.75]    (423,67.61) -- (217.15,65.98) ;
\draw [shift={(217.15,65.98)}, rotate = 180.45] [color={rgb, 255:red, 0; green, 0; blue, 0 }  ,draw opacity=1 ][fill={rgb, 255:red, 0; green, 0; blue, 0 }  ,fill opacity=1 ][line width=0.75]      (0, 0) circle [x radius= 3.35, y radius= 3.35]   ;
\draw [shift={(423,67.61)}, rotate = 180.45] [color={rgb, 255:red, 0; green, 0; blue, 0 }  ,draw opacity=1 ][fill={rgb, 255:red, 0; green, 0; blue, 0 }  ,fill opacity=1 ][line width=0.75]      (0, 0) circle [x radius= 3.35, y radius= 3.35]   ;
\draw [color={rgb, 255:red, 0; green, 0; blue, 0 }  ,draw opacity=0.4 ][line width=1.5]  [dash pattern={on 5.63pt off 4.5pt}]  (423,67.61) -- (255.05,66.98) ;
\draw [shift={(255.05,66.98)}, rotate = 180.21] [color={rgb, 255:red, 0; green, 0; blue, 0 }  ,draw opacity=0.4 ][fill={rgb, 255:red, 0; green, 0; blue, 0 }  ,fill opacity=0.4 ][line width=1.5]      (0, 0) circle [x radius= 4.36, y radius= 4.36]   ;
\draw [shift={(423,67.61)}, rotate = 180.21] [color={rgb, 255:red, 0; green, 0; blue, 0 }  ,draw opacity=0.4 ][fill={rgb, 255:red, 0; green, 0; blue, 0 }  ,fill opacity=0.4 ][line width=1.5]      (0, 0) circle [x radius= 4.36, y radius= 4.36]   ;
\draw [color={rgb, 255:red, 0; green, 0; blue, 0 }  ,draw opacity=0.4 ][line width=1.5]  [dash pattern={on 5.63pt off 4.5pt}]  (338.25,203.34) -- (423,67.61) ;
\draw [shift={(423,67.61)}, rotate = 301.98] [color={rgb, 255:red, 0; green, 0; blue, 0 }  ,draw opacity=0.4 ][fill={rgb, 255:red, 0; green, 0; blue, 0 }  ,fill opacity=0.4 ][line width=1.5]      (0, 0) circle [x radius= 4.36, y radius= 4.36]   ;
\draw [shift={(338.25,203.34)}, rotate = 301.98] [color={rgb, 255:red, 0; green, 0; blue, 0 }  ,draw opacity=0.4 ][fill={rgb, 255:red, 0; green, 0; blue, 0 }  ,fill opacity=0.4 ][line width=1.5]      (0, 0) circle [x radius= 4.36, y radius= 4.36]   ;
\draw [color={rgb, 255:red, 0; green, 0; blue, 0 }  ,draw opacity=1 ][line width=0.75]    (218,54) -- (253.52,53.67) ;
\draw [shift={(256.52,53.64)}, rotate = 539.46] [fill={rgb, 255:red, 0; green, 0; blue, 0 }  ,fill opacity=1 ][line width=0.08]  [draw opacity=0] (10.72,-5.15) -- (0,0) -- (10.72,5.15) -- (7.12,0) -- cycle    ;
\draw    (373,216) -- (340.32,215.78) ;
\draw [shift={(337.32,215.76)}, rotate = 360.39] [fill={rgb, 255:red, 0; green, 0; blue, 0 }  ][line width=0.08]  [draw opacity=0] (10.72,-5.15) -- (0,0) -- (10.72,5.15) -- (7.12,0) -- cycle    ;
\draw [color={rgb, 255:red, 0; green, 0; blue, 0 }  ,draw opacity=0.4 ][line width=1.5]  [dash pattern={on 5.63pt off 4.5pt}]  (167,201) -- (338.25,203.34) ;
\draw [shift={(338.25,203.34)}, rotate = 0.78] [color={rgb, 255:red, 0; green, 0; blue, 0 }  ,draw opacity=0.4 ][fill={rgb, 255:red, 0; green, 0; blue, 0 }  ,fill opacity=0.4 ][line width=1.5]      (0, 0) circle [x radius= 4.36, y radius= 4.36]   ;
\draw [shift={(167,201)}, rotate = 0.78] [color={rgb, 255:red, 0; green, 0; blue, 0 }  ,draw opacity=0.4 ][fill={rgb, 255:red, 0; green, 0; blue, 0 }  ,fill opacity=0.4 ][line width=1.5]      (0, 0) circle [x radius= 4.36, y radius= 4.36]   ;
\draw [color={rgb, 255:red, 0; green, 0; blue, 0 }  ,draw opacity=0.4 ][line width=1.5]  [dash pattern={on 5.63pt off 4.5pt}]  (255.05,66.98) -- (167,201) ;
\draw [shift={(167,201)}, rotate = 123.31] [color={rgb, 255:red, 0; green, 0; blue, 0 }  ,draw opacity=0.4 ][fill={rgb, 255:red, 0; green, 0; blue, 0 }  ,fill opacity=0.4 ][line width=1.5]      (0, 0) circle [x radius= 4.36, y radius= 4.36]   ;
\draw [shift={(255.05,66.98)}, rotate = 123.31] [color={rgb, 255:red, 0; green, 0; blue, 0 }  ,draw opacity=0.4 ][fill={rgb, 255:red, 0; green, 0; blue, 0 }  ,fill opacity=0.4 ][line width=1.5]      (0, 0) circle [x radius= 4.36, y radius= 4.36]   ;
\draw [color={rgb, 255:red, 0; green, 0; blue, 0 }  ,draw opacity=1 ][line width=0.75]    (374.73,203.34) -- (423,67.61) ;
\draw [shift={(423,67.61)}, rotate = 289.58] [color={rgb, 255:red, 0; green, 0; blue, 0 }  ,draw opacity=1 ][fill={rgb, 255:red, 0; green, 0; blue, 0 }  ,fill opacity=1 ][line width=0.75]      (0, 0) circle [x radius= 3.35, y radius= 3.35]   ;
\draw [shift={(374.73,203.34)}, rotate = 289.58] [color={rgb, 255:red, 0; green, 0; blue, 0 }  ,draw opacity=1 ][fill={rgb, 255:red, 0; green, 0; blue, 0 }  ,fill opacity=1 ][line width=0.75]      (0, 0) circle [x radius= 3.35, y radius= 3.35]   ;
\draw [color={rgb, 255:red, 0; green, 0; blue, 0 }  ,draw opacity=1 ][line width=0.75]    (217.15,65.98) -- (167,200) ;
\draw [shift={(167,200)}, rotate = 110.52] [color={rgb, 255:red, 0; green, 0; blue, 0 }  ,draw opacity=1 ][fill={rgb, 255:red, 0; green, 0; blue, 0 }  ,fill opacity=1 ][line width=0.75]      (0, 0) circle [x radius= 3.35, y radius= 3.35]   ;
\draw [shift={(217.15,65.98)}, rotate = 110.52] [color={rgb, 255:red, 0; green, 0; blue, 0 }  ,draw opacity=1 ][fill={rgb, 255:red, 0; green, 0; blue, 0 }  ,fill opacity=1 ][line width=0.75]      (0, 0) circle [x radius= 3.35, y radius= 3.35]   ;
\draw [color={rgb, 255:red, 0; green, 0; blue, 0 }  ,draw opacity=0.4 ]   (362,203.25) -- (423,67.61) ;
\draw [color={rgb, 255:red, 0; green, 0; blue, 0 }  ,draw opacity=0.3 ]   (353.67,202.33) -- (423,67.61) ;
\draw [color={rgb, 255:red, 0; green, 0; blue, 0 }  ,draw opacity=0.2 ][fill={rgb, 255:red, 0; green, 0; blue, 0 }  ,fill opacity=0.2 ]   (348.33,202.67) -- (423,67.61) ;
\draw [color={rgb, 255:red, 0; green, 0; blue, 0 }  ,draw opacity=0.1 ][fill={rgb, 255:red, 0; green, 0; blue, 0 }  ,fill opacity=0.2 ]   (343.29,203) -- (423,67.61) ;
\draw [color={rgb, 255:red, 0; green, 0; blue, 0 }  ,draw opacity=0.4 ]   (230.48,66.3) -- (167,201) ;
\draw [color={rgb, 255:red, 0; green, 0; blue, 0 }  ,draw opacity=0.3 ]   (238.84,66.87) -- (167,201) ;
\draw [color={rgb, 255:red, 0; green, 0; blue, 0 }  ,draw opacity=0.2 ][fill={rgb, 255:red, 0; green, 0; blue, 0 }  ,fill opacity=0.2 ]   (244.16,66.32) -- (167,201) ;
\draw [color={rgb, 255:red, 0; green, 0; blue, 0 }  ,draw opacity=0.1 ][fill={rgb, 255:red, 0; green, 0; blue, 0 }  ,fill opacity=0.2 ]   (249.19,65.77) -- (167,200) ;

\draw (138.84,197.51) node [anchor=north west][inner sep=0.75pt]    {$p_{0}$};
\draw (376.73,208.74) node [anchor=north west][inner sep=0.75pt]    {$p_{1}$};
\draw (428.03,39.28) node [anchor=north west][inner sep=0.75pt]    {$p_{2}$};
\draw (195.03,40.28) node [anchor=north west][inner sep=0.75pt]    {$p_{3}$};
\draw (310.73,205.74) node [anchor=north west][inner sep=0.75pt]    {$p_{1}^*$};
\draw (263.73,36.74) node [anchor=north west][inner sep=0.75pt]    {$p_{3}^*$};

\end{tikzpicture}

    \caption{A parallelogram will limit towards a rhombus under iteration.}
    \label{fig:parallelogram}
\end{figure}
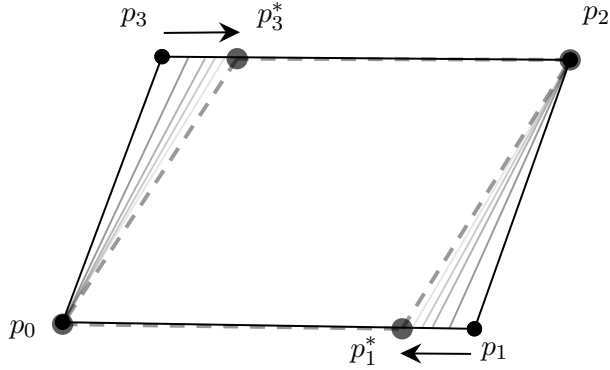

\end{eg}

\section{Application}  \label{sec:application}
In this section we provide an outline of how one can implement the arclength respacing method efficiently, and empirically explore the practical effect it has on polygonal curves approximating a given shape.

\begin{alg} An outline of the implementation of arclength respacing.  An example implementation in \texttt{Mathematica} can be found on the third author's website, \cite{t}.

\label{alg:respacing}

\vskip 5pt
\noindent
\textit{Input:}  Points $\{p_0, \ldots, p_m\}$ representing the vertices of a polygonal curve $C$.

\vskip 5pt
\noindent
\textit{Output:}  Points $\{q_0, \ldots, q_m\}$ representing the arclength respacing $f(C)$ of $C$.
   \vskip 5pt

\begin{enumerate}
    \item  Let $d_0 = 0$ and $d_k = ||p_k - p_{k-1}|| + d_{k-1}$ for $k=1, \ldots m$.  $d_k$ is the piecewise linear arclength distance from $p_0$ to $p_k$.

    \item
    Compute the piecewise linear interpolating function $g:[0,d_m] \rightarrow [0,m]$ for the points $(d_k,k)$, $0 \leq k \leq m$.  This function inverts the arclength measurements, so that $g(d_k) = k$.  Here we require $p_k \neq p_{k+1}$ in order for this inverse to be well defined.
    
   \item
  Compute the piecewise linear interpolating function $h:[0,m] \rightarrow \mathbb{R}^2$ for the discrete curve points $p_0, \ldots, p_m$.
  
  \item  Let $\delta = d_m/m$ and define $q_k = h(g(k \delta))$ for  $k = 0, 1, \ldots, m$.  The points $q_k$ are separated by an arclength distance of $\delta$ along  $C$.  These points are the vertices of $f(C)$.

  \end{enumerate}

\end{alg}

\begin{remark}
Note that in Step 2 of Algorithm \ref{alg:respacing}, the linear interpolation of the inverted arclength distances is key to avoiding inefficient arclength integral computations when performing arclength respacing.  Also note that the distance of $\delta$ in Step 4 can be chosen to fit the application.  For example, one could choose the same $\delta$ across a collection of curves in order to have consistent arclength spacing for curve comparison.
\end{remark}

We now illustrate the effect of respacing on polygonal curves approximating shapes.  Shown in Figure \ref{fig:cat} is a synthetic ``noisy cat'' curve after $0$, $1$, and $5$ applications of respacing.  The points are visually equilateral after only a few iterations.  The ears (and other smaller protrusions) become rounded, but this smoothing effect stabilizes quickly.  The jagged point on the right side of the figure does not smooth out; this is an artifact of the choice of starting point for the respacing.

\begin{figure}[ht]
\centering
\includegraphics[width=0.95\textwidth]{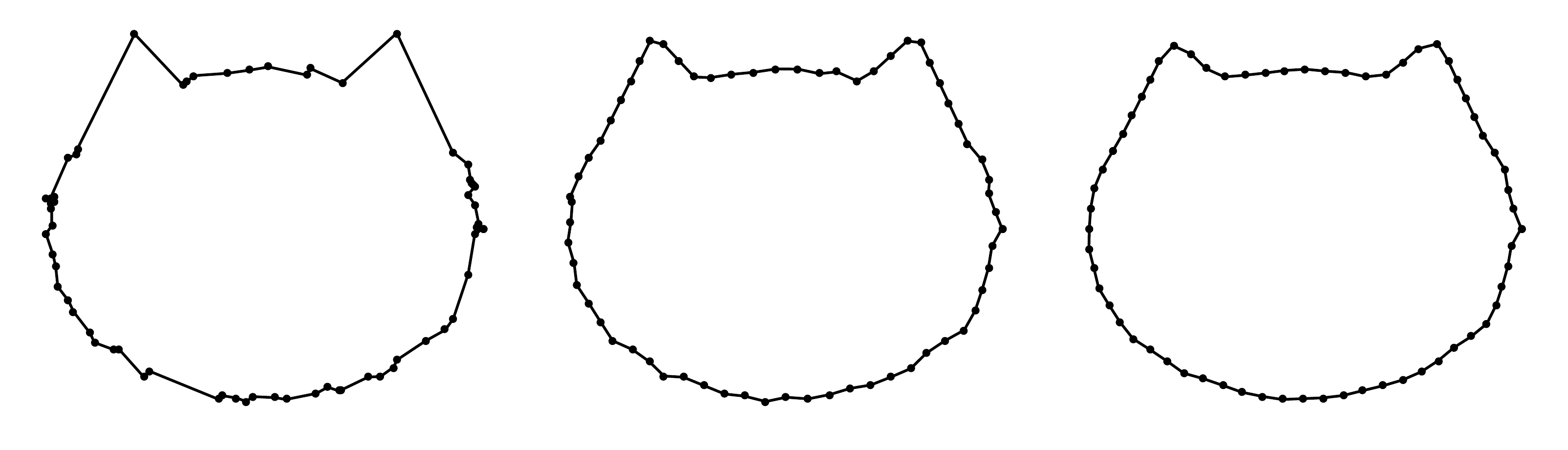}
\caption{A ``noisy cat'' polygonal curve after $0$, $1$, and $5$ iterations of arclength respacing. }
\label{fig:cat}
\end{figure}

The noisy cat curve consists of $65$ points, with $x,y$ values in the range $[-1,1]$.  We compute some basic statistics at the $n^{th}$ iteration to quantify the effect of the arclength respacing: the standard deviation $\sigma^n$ of the collection of distances $d^n_k = ||p^n_{k} - p_{k-1}^n||$, $k=1, \ldots, m$,  the maximum interpoint distance $\max^n =$ $\displaystyle \max_{1 \leq k \leq n} d_k^n$, and the minimum interpoint distance $\min^n =$ $\displaystyle \min_{1 \leq k \leq n} d_k^n$.  These statistics are shown in Figure \ref{tab:cat} for various iterations.  The standard deviation $\sigma^n$ appears to decrease exponentially with a factor of about $0.54$.  After approximately 15 iterations, the standard deviation is $0$ and the maximum and minimum are equal up to first five decimal places.

\begin{figure}[h]

\centering
\begin{tabular}{|l|l|l|l|l|l|l|l|}
\hline
$\bm{n}$        & 0 & 1 & 2 & 3 & 5 & 10 & 15 \\ \hline
$\bm{\sigma^n}$ & 0.127073 & 0.01431 &  0.00342 & 0.00117 & 0.00039 & 0.00002 & 0.00000
    \\ \hline
$\bm{\sigma^n/\sigma^{n-1}}$   & -  & 0.112646 & 0.23937 & 0.34217 & 0.66789 & 0.53647 & 0.53571   \\ \hline
$\bm{\max^n}$   & 0.65736 & 0.11137 & 0.10436 & 0.10274 & 0.10217 & 0.10199 & 0.10198  \\ \hline
$\bm{\min^n}$   & 0.010769 &  0.02796 & 0.08312 & 0.09420 & 0.09923 & 0.10184 & 0.10197  \\ \hline
\end{tabular}
\caption{Measuring the respacing of the ``noisy cat'' polygonal curve.}
\label{tab:cat}
\end{figure}

We next apply arclength respacing to the outline of a jigsaw puzzle piece obtained from image segmentation, shown in Figure \ref{fig:puzzle}.  This jigsaw curve consists of $400$ points, with $x,y$ values in $[-6,6]$.  The vertices have been normalized so that mean interpoint distances are comparable to the ``noisy cat'' example, with $\delta \approx 0.1$.  Typical methods for simplifying output in segmentation, such as the Douglas-Peucker algorithm, \cite{dp}, result in a very unevenly spaced outline curve.  One iteration of arclength respacing produces a curve which appears very close to equilateral.   

\begin{figure}[h]
\centering
\includegraphics[width=0.85\textwidth]{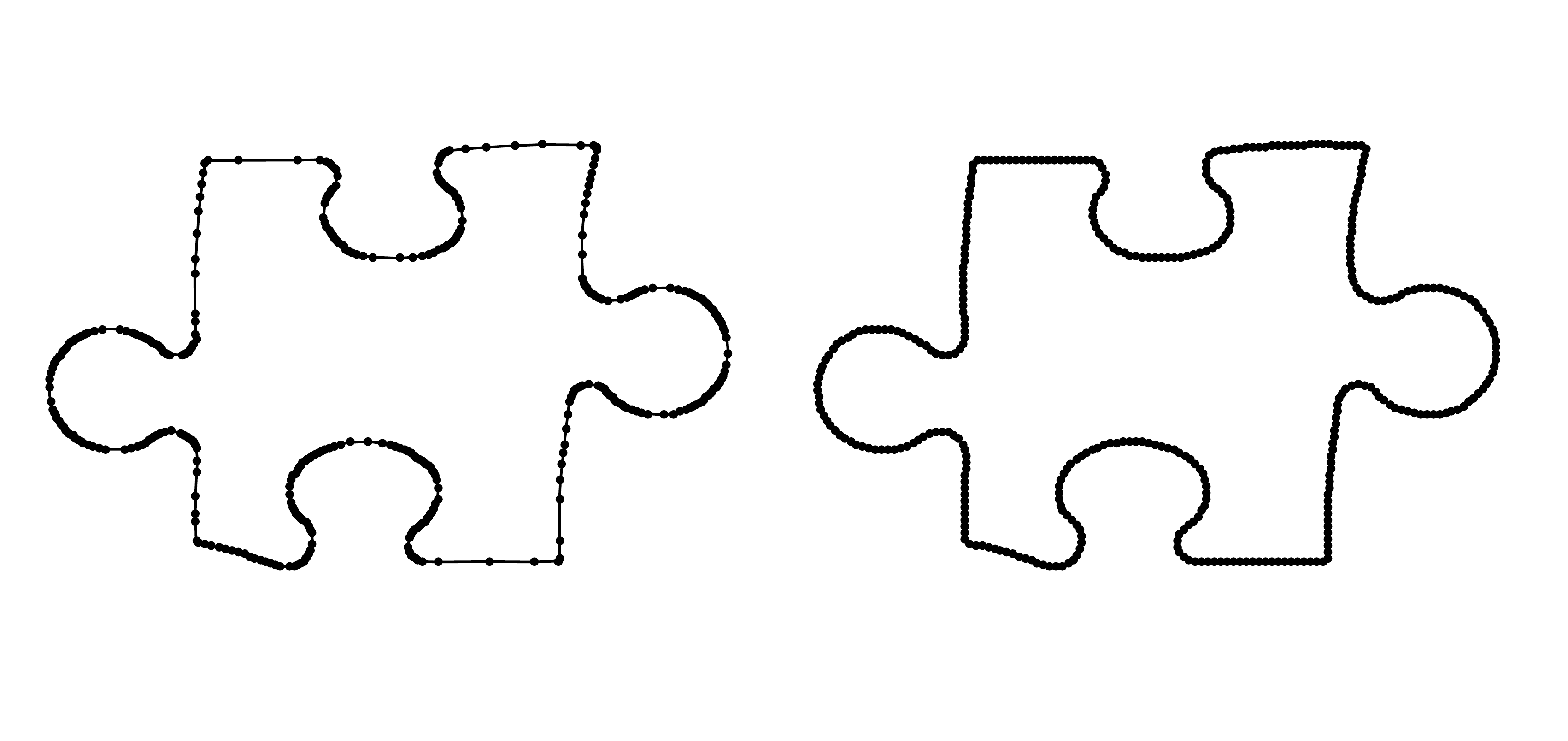}
\caption{An outline of a puzzle piece after $0$ and $1$ iteration. }
\label{fig:puzzle}
\end{figure}

Shown in Figure \ref{tab:puzzledata} are the basic statistics for the arclength respacing iteration applied to this puzzle piece curve, including  standard deviation $\sigma^n$, maximum $\max^n$, and minimum $\min^n$ of interpoint distances at iteration $n$.  Once again, the standard deviation $\sigma^n$ appears to decrease exponentially, and the standard deviation is $0$ and the maximum and minimum are equal up to first five decimal places after approximately 15 iterations.

\begin{figure}[h]

\centering
\begin{tabular}{|l|l|l|l|l|l|l|l|}
\hline
$\bm{n}$        & 0 & 1 & 2 & 3 & 5 & 10 & 15  \\ \hline
$\bm{\sigma^n}$ & 0.11327 & $0.00123$ & $.00038$  & $0.00011$ & $0.00001$ & $0.00000$
 & $0.00000$  \\ \hline
$\bm{\sigma^n/\sigma^{n-1}}$   & -  &0.01084   & 0.31176  & 0.29806  & 0.29953  & 0.40801  & 0.58714   \\ \hline
$\bm{\max^n}$   & 0.96667   & 0.10400  &0.10369   & 0.10362   & 0.10358  & 0.10358  & 0.10358   \\ \hline
$\bm{\min^n}$   &   0.03018 & 0.08805& 0.09861& 0.10218   & 0.10341  & 0.10357  & 0.10357   \\ \hline
\end{tabular}
\caption{Measuring the respacing of a puzzle piece polygonal curve.}
\label{tab:puzzledata}
\end{figure}

\section{Conclusion}  \label{sec:conclusion}

This paper introduces an arclength respacing method for polygonal curves and establishes general facts about the behavior of polygonal curves under iteration of this respacing.  There are several interesting directions for further investigation.

The impetus for this research was the application of arclength respacing to problems in computer vision.  For this application it would be very useful to better understand the rate of convergence to an equilateral polygon and the  smoothing effect, the latter being especially apparent in Figure \ref{fig:cat}.  One might also consider replacing the underlying polygonal curve with a higher order interpolating function, as is done in \cite{ho}, to affect rates of convergence and smoothing.

Simple, concrete examples of limiting polygonal curves are provided in Section \ref{sec:examples}, but we found it generally difficult to determine limiting curves exactly.  Further investigation may discover ways to obtain information about the limiting curves, e.g. bounds for final locations of vertices or knowledge of whether the vertices will converge to a single point.  Additionally, one could study stability --  ways in which perturbations to the initial vertex configuration affect the limiting curve.

Finally, the arclength respacing iteration bears a resemblance to the \textit{pentagram map}, \cite{s}, an operation defined on convex polygons with many interesting properties, including discrete integrability and a continuum limit corresponding to the Boussinesq equation,  \cite{vst,b}.  It would be worthwhile to investigate this resemblance.  A first step could be understanding the continuum limit of arclength respacing; does the infinitesimal motion of the vertices produce a non-trivial curve flow?  This curve flow would likely be non-local since the respacing process is non-local.


\section*{Acknowledgements}
$\indent$ We would like to thank the Carleton College Towsley Endowment, the Louis Stokes Alliances for Minority Participation, the Carleton Summer Science Fellowship with funding from
Carleton College S-STEM (NSF 0850318 and 156018), and North Star STEM Alliance, an LSAMP alliance
(NSF 1201983), for providing the funding that made this research possible.   We would also like to thank Peter Olver for helpful comments and advice.


\end{document}